\documentclass{article}

\usepackage[a4paper, scale = .75]{geometry}

\usepackage[T1]{fontenc}           
\usepackage[utf8]{inputenc}        

\usepackage{amsmath, amsfonts, amssymb, amsthm, empheq}  
\usepackage{siunitx}                                     

\usepackage{graphicx}             
\usepackage{float}                
\usepackage{caption}              
\usepackage{subcaption}           

\usepackage{booktabs}             
\usepackage{array}                

\usepackage[hidelinks]{hyperref}             
\usepackage{cleveref}             
\usepackage{cite}                 

\usepackage{microtype}            

\usepackage{todonotes}            

\usepackage{xcolor}
\definecolor{darkblue}{rgb}{0,0,.45}
\definecolor{darkblue2}{rgb}{0,0,.65}
\definecolor{darkred}{rgb}{0.6,0,0}
\definecolor{darkred2}{rgb}{0.85,0,0}
\definecolor{darkgreen}{rgb}{0,0.5,0}
\definecolor{darkgreen2}{rgb}{0,0.75,0}
\definecolor{darkgray}{rgb}{0.15,0.15,0.15}

\def\electrons{\text{n}}
\def\holes{\text{p}}

\def\ions{\text{a}}

\theoremstyle{definition}
\newtheorem{theorem}{Theorem}
\newtheorem{lemma}[theorem]{Lemma}
\newtheorem{definition}[theorem]{Definition}
\newtheorem{proposition}[theorem]{Proposition}
\newtheorem{corollary}[theorem]{Corollary}
\newtheorem{remark}[theorem]{Remark}
\newtheorem{assumption}[theorem]{Assumption}
\newtheorem{example}[theorem]{Example}

\makeatletter
\numberwithin{equation}{section}
\numberwithin{figure}{section}
\numberwithin{theorem}{section}
\makeatother

\renewcommand{\appendix}{%
  \par
  \setcounter{section}{0}%
  \setcounter{subsection}{0}%
  \gdef\thesection{\Alph{section}}%
}

\newcommand*\samethanks[1][\value{footnote}]{\footnotemark[#1]} 

\title{Existence of solutions and uniform bounds for the stationary semiconductor equations with generation and ionic carriers}

\author{
    Dilara Abdel\thanks{Weierstrass Institute for Applied Analysis and Stochastics (WIAS), 10117 Berlin, Germany}
    \and
    Alain Blaustein\thanks{Univ. Lille, CNRS, Inria, UMR 8524 - Laboratoire Paul Painlevé, F-59000 Lille, France}
    \and
    Claire Chainais-Hillairet\samethanks[2]
    \and
    Maxime Herda\samethanks[2]
    \and
    Julien Moatti\thanks{Bordeaux INP, Univ. Bordeaux, CNRS, Inria, IMB, UMR 5251,
    F-33400 Talence, France}
}

\date{}

\begin{document}

\maketitle

\begin{abstract}
We consider a stationary drift-diffusion system with ionic charge carriers and external generation of electron and hole charge carriers. This system arises, among other applications, in the context of semiconductor modeling for perovskite solar cells. Thanks to truncation techniques and iterative energy estimates, we show the existence and uniform upper and lower bounds on the solutions. The dependency of the bounds on the various parameters of the model is investigated numerically on physically relevant test cases.

\bigskip


\medskip


\end{abstract}

\tableofcontents

\section{Introduction}

The mathematical modeling of charge transport in semiconductors is a classical subject, central both to mathematical analysis of partial differential equations (PDEs) and to the design of electronic and optoelectronic devices.
In conventional semiconductor devices, drift-diffusion models describe the motion of two mobile charge carriers—electrons and holes—coupled through the electrostatic potential~\cite{van1950theory, Mock83, gajewski86, Markowich86}.
However, in emerging technologies such as perovskite solar cells, additional physical effects must be incorporated into the modeling framework~\cite{calado2016evidence, courtier2018fast,  Abdel2021Model}.
A key example is ionic transport: in addition to electronic carriers, the migration of ionic vacancies, despite being slow, influences device operation \cite{calado2016evidence}.
Unlike standard models for solar cells, this extra species introduces nonlinearities associated with volume exclusion effects, which are not captured~\cite{Abdel2021Model} by the usual Boltzmann or Fermi-Dirac statistics~\cite{gajewski1989semiconductor} employed for electrons and holes.
These mobile ionic species are also crucial in other modern semiconductor applications such as memristor modeling \cite{jourdana2023three, Spetzler.2024, herda2025charge, abdel2025schottky}.
Another essential feature of solar cell modeling is photogeneration, which acts as an external source of carriers and must be considered alongside the intrinsic recombination-generation mechanisms.
Beyond photovoltaic energy conversion, charge carrier generation induced by external radiation is also relevant in light sensing with photodiodes \cite{Pisarenko2016, Hu2017} and non-destructive diagnostic techniques such as the laser beam induced current (LBIC)~\cite{BFI:93,FaIt:96, Yin2009LBIC, Redfern2005LBIC} or the lateral photovoltage scanning (LPS) method \cite{Farrell2021,Kayser2021}.
In all these cases, the generation strongly impacts charge carrier transport.

Motivated by these applications we are interested in this paper in the analysis of a generic charge transport system which includes a mobile ionic species and generation effects. There has been several recent works concerning the analysis of time dependent version of related systems \cite{jourdana2023three, herda2025charge, glitzky2025uniqueness, abdel2025analysis}. However, the key estimates of the previous papers are non-uniform in time so the analysis of the steady equations does not follow. Henceforth, the goal of this paper is to analyze the corresponding stationary system of equations, derive explicit bounds on solutions and investigate the influence of parameters on these bounds in physically relevant test cases.

\subsection{Stationary drift-diffusion equations with ionic carriers and generation}

  In a bounded domain $\Omega \subset \mathbb{R}^d$, $d \geq 1$, we consider a non-dimensionalized system of continuity equations
\begin{subequations} \label{eq:model-dimless}
    \begin{alignat}{2}
       -\nabla\cdot \left(n_{\electrons} \nabla v_{\electrons}\right) &= G-R(n_{\electrons}, n_{\holes}), \quad && x\in\Omega, \label{eq:model-cont-eq-n-dimless}\\
         -\nabla\cdot \left(n_{\holes} \nabla v_{\holes}\right) &= G-R(n_{\electrons}, n_{\holes}), \quad && x\in\Omega, \label{eq:model-cont-eq-p-dimless}
    \end{alignat}
    which are self-consistently coupled to the Poisson equation with rescaled Debye length $\lambda>0$ reading
    \begin{equation}\label{eq:model-poisson-dimless}
        - \lambda^2\Delta\psi =  z_{\electrons}n_{\electrons} + z_{\holes}n_{\holes} + z_{\ions}n_{\ions} + C , ~ \quad {x}\in{\Omega}.
    \end{equation}
    Here, $\psi$ denotes the electric potential and $n_\alpha\equiv n_\alpha(x)$, $\alpha \in \{ \electrons, \holes, \ions\}$ the charge carrier density, respectively for electrons, holes and ions.
    The subscript $\alpha = \ions$ is chosen for the index related to ions because in perovskite solar cell applications, this species corresponds usually to \emph{anionic} vacancies. The quantity $G\equiv G(x)$ denotes an external generation rate, $R(n_{\electrons}, n_{\holes})$ is the electron-hole recombination-generation, and $C\equiv C(x)$ is the given doping density.
    For each charge carrier species, $z_\alpha$ is the charge number with $z_\electrons = - z_\holes = -1$ and $z_\ions \in \mathbb{Z}$, and $v_\alpha\equiv v_\alpha(x)$ is related to the quasi-Fermi potential $z_\alpha^{-1}v_\alpha$ of the carrier species $\alpha$. Without distinction, we shall call $v_\alpha$ quasi-Fermi potentials in the following. The potentials and densities are related by the state equation
    \begin{equation} \label{eq:state-eq-stationary}
            n_\alpha = \mathcal{F}_\alpha \left(v_\alpha- z_\alpha \psi\right),\quad \text{and}\quad v_\alpha=\mathcal{F}_{\alpha}^{-1}(n_\alpha)+z_\alpha\psi,
    \end{equation}
    where $\mathcal{F}_\alpha$ is called statistics function.
    The statistics function and the recombination-generation term will be precised in Section~\ref{sec:mainresult}. The ionic vacancies are isolated in the semiconductor and at thermodynamic equilibrium in the present model meaning that $v_\ions$ is constant and determined by a given total mass $M_\ions>0$ through the relation
    \begin{equation}\label{eq:model-equilibrium ions}
    \nabla v_\ions = 0,\quad\text{and}\quad \int_\Omega n_\ions\,\mathrm{d} x = M_\ions.
    \end{equation}
    Thus, $n_{\ions}$ should be seen as a function of the electric potential only, making the Poisson equation \eqref{eq:model-poisson-dimless} nonlinear with respect to $\psi$ for given sources $n_\electrons$, $n_\holes$ and $C$.
\end{subequations}

The system is supplemented with Dirichlet boundary conditions on a part $\Gamma^D$ of the boundary $\Gamma=\partial \Omega$ and on Neumann boundary conditions on the complementary part $\Gamma^N$:
    \begin{subequations}\label{BC:model_bipolar_stationary}
    \begin{align}
    	& n_\electrons= n_\electrons^D, n_\holes= n_\holes^D, \psi=\psi^D \mbox{ on } \Gamma^D,\label {BC-dir:model_bipolar_stationary}\\
& \nabla n_\electrons\cdot \nu= \nabla n_\holes\cdot\nu= \nabla  \psi\cdot\nu=0\mbox{ on } \Gamma^N,\label {BC-neum:model_bipolar_stationary}
\end{align}
\end{subequations}
where $\nu$ denotes the exterior normal vector and $n_\electrons^D$, $n_\holes^D$ are given functions which will be defined in Assumption~\ref{main_ass}.

By combining the equations of \eqref{eq:model-dimless} the system can be recast into the drift-diffusion form
\begin{subequations} \label{eq:model-drift-diff}
    \begin{align}
       -\nabla\cdot \left( D_{\electrons}(n_{\electrons})\nabla n_{\electrons} - n_{\electrons}\nabla\psi\right) = G-R(n_{\electrons}, n_{\holes}) , && x\in\Omega, \label{eq:model-drift-diff-n}\\
         -\nabla\cdot \left( D_{\holes}(n_{\holes})\nabla n_{\holes} + n_{\holes}\nabla\psi\right) = G-R(n_{\electrons}, n_{\holes}), && x\in\Omega, \label{eq:model-drift-diff-p}\\
          - \lambda^2\Delta\psi = n_{\holes} - n_{\electrons} + z_{\ions}\mathcal{F}_\ions \left(v_\ions- z_\ions \psi\right) + C , &&{x}\in{\Omega}.\label{eq:model-drift-diff-Poisson}
    \end{align}
\end{subequations}
where the nonlinear diffusion coefficients are given by
\begin{equation}\label{def:D}
D_\alpha: x\in (0,+\infty)\mapsto  x({\mathcal F}_\alpha^{-1})'(x).
\end{equation}
Observe that this framework includes the case $z_\ions=0$ for which \eqref{eq:model-drift-diff} reduces to a stationary two-species system with generation and a linear Poisson equation.

The early seminal contributions to the analysis of two-species stationary drift-diffusion equations can be found in \cite{Mock83, Markowich86, groger1987steady}.  In~\cite{groger1987steady}, Gröger shows the existence of solutions for a stationary drift-diffusion system with general statistics and recombinations, but without generation term $G$. His analysis relies on truncations techniques that we shall use in the present analysis. However, the uniform bounds on the solution cannot be adapted in the presence of an external generation term $G$.
In the context of LPS method, the recent~\cite{Ali2025} is concerned in a two-species drift-diffusion model with generation. The difficulty to deal with arbitrary large generation term is mentioned in \cite[Section 1]{Ali2025}, and the analysis developed in the paper is instead based on a perturbative approach leading to a simplified system.

Elements of analysis of semiconductor model with external generation term can be found in the context of non-destructive control of semiconductor devices. In these applications, the optical generation term correspond to the energy induced by a laser beam. In~\cite{BFI:93,FaIt:96}, drift-diffusion systems with external generation and recombination terms are considered. The existence of solutions is proved using a fixed-point argument, and is valid under a smallness assumption on the external generation term. In~\cite{CuFa:96}, a similar system is studied in one-dimension, for which, under some simplification, analytical expressions of solutions are given.
 In \cite{FaWu:01} a drift-diffusion system coupled with an energy term (taking into account thermal effects) and external generation is considered. Using truncation techniques, $L^\infty$ bounds and a topological fixed-point argument, the existence of solutions is proved without any restriction on the intensity of the laser beam. However the analysis is limited to the two-species model with linear diffusion (\emph{i.e.} $D_\alpha = 1$ or equivalently $\mathcal{F}_\alpha = \exp$).

The present paper is concerned with the general three species model \eqref{eq:model-dimless}. The analysis includes physically realistic statistics functions leading to nonlinear diffusion in the drift-diffusion form \eqref{eq:model-drift-diff}, realistic recombination-generation mechanisms and external generation term $G$ without any size restriction. The latter point constitutes the major mathematical difficulty in the analysis.
The main contributions of the present paper are the proof of existence of solutions to the system \eqref{eq:model-dimless}, the derivation of uniform $L^\infty$ bounds  on solutions and bounds from below on the density, as well as the numerical investigation of the dependencies of those bounds on the various parameters of the equation in physically relevant settings.

\subsection{Main results}\label{sec:mainresult}

Let us start with some notation. With $|\cdot|$ we denote without distinction the Euclidean norm in $\mathbb{R}^n$,  the Lebesgue measure of a set in $\mathbb{R}^d$ or the $(d-1)$-dimensional Hausdorff surface measure of a set. The notation ${\mathbf 1}_A$ denotes the indicator function of the set $A$. For Lebesgue spaces $L^p(\Omega)$ we use the notation $\|\cdot\|_{L^p}$ for the canonical norm and $p' = p/(p-1)$ denotes the conjugate Lebesgue exponent. For the Sobolev space $H^1(\Omega)$ we use the notation $\|\cdot\|_{H^1}$ for the canonical norm. Along the proofs, $K$ denotes a positive constant which may change from line to line. Before stating the main theorem, let us gather the assumptions on the data.
\begin{assumption}\label{main_ass}
 The domain $\Omega\subset {\mathbb R}^d$ is bounded, connected and Lipschitz. The boundary $\partial \Omega$ is divided into disjoint Lipschitz parts $\Gamma^D$, $\Gamma^N$ such that $\overline{\Gamma^D\cup\Gamma^N}=\partial \Omega$ and $|\Gamma^D|>0$. The electron and hole densities $n_\electrons$ and $n_\holes$ are related to the electric and quasi-Fermi potential through \eqref{eq:state-eq-stationary} and for $\alpha \in\{\electrons,\holes\}$, we assume
    \begin{equation} \label{hyp:statistics-n-p}
        \left\{
        \begin{aligned}
            &\mathcal{F}_\alpha : \mathbb{R} \rightarrow (0,\infty)\text{ is a }  C^1\text{- diffeomorphism};\\[.5em]
            &0 < \mathcal{F}_\alpha'(\eta) \leq \mathcal{F}_\alpha(\eta) \leq \exp(\eta), \quad \forall\,\eta \in \mathbb{R} .
        \end{aligned}
        \right.
    \end{equation}
The ionic density $n_{\ions}$ is related to the electric and quasi-Fermi potentials through \eqref{eq:state-eq-stationary} and we assume
\begin{equation} \label{hyp:statistics-a}
	\left\{
	\begin{aligned}
    &S_\ions\coloneqq\sup\mathcal{F}_\ions< +\infty\,; \\[.2em]
		&\mathcal{F}_\ions : \mathbb{R} \rightarrow (0\,,\,S_\ions)\text{ is a }  C^1\text{- diffeomorphism};\\[.5em]
		&0 < \mathcal{F}_\ions'(\eta) \leq \mathcal{F}_\ions(\eta) \leq \exp(\eta), \quad\forall\,\eta \in \mathbb{R}.
	\end{aligned}
	\right.
\end{equation}
Observe that the boundedness of the image of $\mathcal{F}_{\ions}$ reflects that the density $n_\ions$ is \emph{a priori}  bounded.
 The Dirichlet boundary conditions are such that  $\psi^D, v_{\alpha}^D\in H^1(\Omega)\cap L^{\infty}(\Omega)$ and $n_\alpha^D$ is defined by  $n_\alpha^D={\mathcal F}_{\alpha}(v_{\alpha}^D-z_\alpha \psi^D)$ for $\alpha\in\{\electrons,\holes\}$. Observe that because of \eqref{hyp:statistics-n-p}, the boundary densities are bounded from above and below and we define
\begin{equation}\label{def:ND}
N^D= \max(\Vert  n_\electrons^D\Vert_{L^\infty}, \Vert  n_\holes^D\Vert_{L^\infty}).
\end{equation}
The total mass $M_\ions$ of the ions in \eqref{eq:model-equilibrium ions} is assumed to satisfy the compatibility condition
\begin{equation}\label{Mass:compatibility}
0\,< \,M_{\ions}\,< \, |\Omega|\,S_\ions\,.
\end{equation}
Concerning the external source terms, the doping profile and generation term are assumed to be such that there is $p\in[1,\infty]$ satisfying $p>d/2$ and
\begin{equation}\label{hyp:CG}
C\in L^{p}(\Omega)\,,\quad G\in L^p(\Omega)\,,\quad G(x)\geq 0\ \text{almost everywhere in } \Omega.
\end{equation}
The recombination-generation term takes the generic form
\begin{align}\label{eq:recombination-standard}
    R(n_{\electrons}, n_{\holes})= r(n_{\electrons}, n_{\holes})n_\electrons n_\holes (1-\exp(-{\mathcal F}_{\electrons}^{-1}(n_\electrons)-{\mathcal F}_{\holes}^{-1}(n_\holes))),
\end{align}
 with the rate satisfying for some constant $r_0\geq0$,
 \[
 0\leq r(n_{\electrons}, n_{\holes}) \leq r_0.
 \]
 The recombination term can be rewritten equivalently
 \begin{equation*}
 R(n_{\electrons}, n_{\holes})= \tilde{r}(n_{\electrons}, n_{\holes})(\exp({\mathcal F}_{\electrons}^{-1}(n_\electrons)+{\mathcal F}_{\holes}^{-1}(n_\holes))-1),
 \end{equation*}
 with the modified rate $\tilde{r}(x,y) = r(x,y) xy \exp(-{\mathcal F}_{\electrons}^{-1}(x)-{\mathcal F}_{\holes}^{-1}(y))$ still satisfying
 \begin{equation}\label{hyp:rate}
 0\leq \tilde{r}(n_{\electrons}, n_{\holes}) \leq r_0,
 \end{equation}
 in virtue of the hypothesis \eqref{hyp:statistics-n-p}. We will use this form of $R$ in the proofs.
\end{assumption}

Let us give examples of statistics functions and recombination-generation rates which satisfy Assumption~\ref{main_ass}.
\begin{example}[Statistics functions] Typically, for electrons and holes one considers either Boltzmann statistics
\[
\mathcal{F}(\eta) = e^\eta,
\]
which yields linear diffusion in \eqref{eq:model-drift-diff}, or Fermi-Dirac of order 1/2 statistics
\begin{align} \label{eq:FD12}
    {\mathcal F}(\eta) = \frac{2}{\sqrt{\pi}}\int_0^\infty\frac{\xi^{1/2}}{\exp(\xi-\eta) + 1} \, \mathrm{d}\xi.
\end{align}
For ionic vacancies, one typically chooses a Blakemore type statistics \cite{Abdel2021Model} writing
\begin{align} \label{eq:Blakemore}
    \mathcal{F}(\eta) = \frac{S_\ions}{e^{-\eta}+1}.
\end{align}

\end{example}

\begin{example}[Recombination-generation rates] The hypotheses made on the recombination-generation rate include the radiative
\[
r_\text{rad}(n_{\electrons}, n_{\holes}) = r_{0, \text{rad}},
\]
with constant rate $r_{0, \text{rad}}$ and trap-assisted Schockley-Read-Hall (SRH) recombination
\[
r_\text{SRH}(n_{\electrons}, n_{\holes}) = \frac{1}{\tau_\electrons (n_\electrons+n_{\electrons,\tau}) + \tau_\holes(n_\holes+n_{\holes,\tau})},
\]
where $\tau_\electrons$, $\tau_\holes$ are the dimensionless carrier lifetimes and $n_{\electrons,\tau}$, $n_{\holes,\tau}$ the dimensionless reference carrier densities.
\end{example}
\begin{example}[External generation]
    Depending on the application, various models for the external optical generation rate are used.
    In solar cell simulations, a common choice is an exponentially decaying profile,
    \begin{align*}
        G(x) = G_0 \exp(-(x-y)\cdot e)\mathbf{1}_E(x), \qquad x = (x_1,\dots,x_d) \in \Omega,
    \end{align*}
    representing absorption along the illumination direction $e$ starting at $y\in E$ in the subdomain $E\subset \Omega$.
    For non-destructive diagnostic scanning techniques, a localized Gaussian generation profile is typical,
    \begin{align*}
        G(x) =
        \exp \bigl( -  \sum_{i=1}^d (x_i - \mu_i)^2 \bigr),
        \qquad x \in \Omega,
    \end{align*}
    centered at the laser beam position $\mu\in\Omega$.
    Both profiles and their associated applications will be illustrated
    in the numerical examples in \Cref{sec:Sim}.
\end{example}

Let us now state the main result of the paper.
\begin{theorem}\label{thm:existence_weak_solution} Under Assumption~\ref{main_ass}, there exists a weak solution \[(v_\ions, v_\electrons,v_\holes,\psi)\in\mathbb{R}\times\left(H^1(\Omega)\cap L^\infty(\Omega)\right)^3\] to the system \eqref{eq:model-dimless} in the sense of Definition~\ref{def:weak-solution}. Moreover, there are explicit positive constants $\overline{M}_\alpha, \overline{M}_\psi, \overline{N}, \underline{N}, \overline{N}_\ions, \underline{N}_\ions$ and $K$ such that for any weak solution to the system \eqref{eq:model-dimless},
\[
|v_\alpha|\leq \overline{M}_\alpha\quad\text{and}\quad  |\psi|\leq \overline{M}_\psi,
\]
for $\alpha\in\{\electrons, \holes, \ions\}$, almost everywhere in $\Omega$. The corresponding densities $n_\electrons$,  $n_\holes$, $n_\ions\in H^1(\Omega)\cap L^\infty(\Omega)$ satisfy \eqref{eq:model-drift-diff} in the weak sense as well as the uniform bounds
\[
0<\underline{N}\leq n_\electrons\leq \overline{N}\,,\quad 0<\underline{N}\leq n_\holes\leq \overline{N}\quad\text{and}\quad 0<\underline{N}_\ions\leq n_\ions\leq\overline{N}_\ions<S_\ions,
\]
almost everywhere in $\Omega$. Moreover, 
\[
\|\nabla v_\electrons\|_{L^2} + \|\nabla v_\holes\|_{L^2} + \|\nabla n_\electrons\|_{L^2} + \|\nabla n_\holes\|_{L^2} + \|\nabla n_\ions\|_{L^2}\leq K.
\]
All constants are explicit and depend on the data only through $\|G\|_{L^p}$, $\|C\|_{L^p}$, $\|\psi^D\|_{H^1\cap L^\infty}$, $p$, $r_0$, $N^D$, $\lambda$, $M_\ions$, $|z_\ions|$, $\mathcal{F}_\alpha$ for $\alpha\in\{\electrons, \holes, \ions\}$ and the domain $\Omega$.
\end{theorem}

The proof of Theorem~\ref{thm:existence_weak_solution} relies on two main steps.

First, we derive uniform bounds on weak solutions of \eqref{eq:model-dimless} in Section \ref{sec:apriori}. The main mathematical difficulty concerns the derivation of \emph{a priori} upper bounds on the densities $n_\electrons$ and $n_\holes$ and is due to the external generation term $G$. In order to derive these bounds we use an iterative energy argument reminiscent of the work of Stampacchia and De Giorgi, see \cite{degiorgi1957sulla, Stampacchia_1963, Stampacchia_1965}. We also refer to Vasseur \cite{vasseur2016Giorgi} and references therein for a modern perspective and to \cite{FaWu:01} for similar techniques used on a related drift-diffusion system.

Second, we build in Section \ref{sec:existenceproof} a solution to the system starting from an approximate system with truncated nonlinearities. This type of approximation is well-known for drift-diffusion systems, see for instance \cite{groger1987steady, Markowich86, herda2025charge, abdel2025analysis}. The existence for the truncated system follows from the Leray-Schauder theorem as well as a minimization  argument for the nonlinear Poisson equation. On the truncated system we show that the uniform bounds also hold which implies that their solution coincide with those of \eqref{eq:model-dimless} for large enough truncation threshold.

We point out that Theorem \ref{thm:existence_weak_solution} guarantees explicit bounds for the electric current 
\begin{equation}\label{outward:current}
	I_B\,=\,\int_{B}(j_\electrons + j_\holes)\cdot \nu \,\mathrm{d}\Gamma\,,\quad \textrm{where}\quad j_\alpha\,=\, - z_\alpha n_\alpha \nabla v_\alpha\,,\quad \alpha \in \{\electrons, \holes\}\,,
\end{equation}
flowing out of the system through an ohmic contact, that is, a boundary component $B$ which satisfies
\begin{equation}\label{ohmic:contact}
B\subset \Gamma^D\,,\quad \overline{B}\cap\overline{B^c\cap \Gamma^D}=\emptyset\,.
\end{equation}
This is of interest because it provides a rigorous framework in the context of laser beam induced current techniques~\cite{BFI:93,FaIt:96, Yin2009LBIC, Redfern2005LBIC}, as detailed in Section \ref{sec:Sim-LBIC}. We formalize this consequence in the following result.
\begin{corollary}\label{coro:bound:current}
    Consider an ohmic contact $B$, namely a subset of the boundary satisfying \eqref{ohmic:contact}. For any weak solution $(v_\ions, v_\electrons,v_\holes,\psi)$ to  \eqref{eq:model-dimless}, the outward current $I_B$ in \eqref{outward:current} is well defined. Furthermore, there exists a constant $K$ depending on the data as in Theorem~\ref{thm:existence_weak_solution} such that
    \[\left|I_B
    \right|\leq K\,.\]
\end{corollary}
This result is a direct consequence of Theorem \ref{thm:existence_weak_solution}. We detail its proof in Section \ref{proof:cor}.

As a final comment, let us mention that uniqueness of solutions to \eqref{eq:model-dimless} is known to be false in general for the system \eqref{eq:model-dimless} even in simple settings, see \cite{alabau1995uniqueness} and references therein.

The analytical results of Theorem~\ref{thm:existence_weak_solution} are complemented in Section~\ref{sec:Sim} with physically motivated numerical simulations.
In two representative applications to study how the generation rate $G$, the Debye length $\lambda$, and the doping profile $C$ influence carrier densities and potential profiles. For these simulations, we employ a stationary version of the time-implicit two-point flux finite volume scheme introduced in \cite{Abdel2023Existence}. While this is outside of the scope of the present paper, we mention that the techniques developed here could be adapted to the finite volume numerical discretizations of the system to show existence and uniform bounds on the approximate solution with uniform dependency on the mesh size. We refer to the work of the third and fourth author of the present paper in \cite{chainais2020linf} (see also \cite{chainais2017finite}) for an example in the case of a simpler linear steady convection diffusion system with mixed boundary conditions.

\subsection{Outline}

Section~\ref{sec:existence} is dedicated to the proof of Theorem~\ref{thm:existence_weak_solution}. In Section~\ref{sec:apriori}, we derive uniform upper and lower bounds for weak solutions of \eqref{eq:model-dimless} using an iterative energy argument for the upper bounds on the densities and monotonicity / maximum principle arguments for the rest of the bounds. In Section~\ref{sec:existenceproof}, we build a solution to the system from an approximate version with truncated nonlinearities. Then, Section~\ref{sec:Sim} is dedicated to numerical investigation of the bounds on two test cases: a three-layer perovskite solar cell with anion vacancies and a two-dimensional laser-beam-induced-current setup for the a classical two-species drift-diffusion system. Finally we recall in appendix a technical lemma due to Stampacchia which is crucial in the iterative energy argument.

\section{Existence of weak solutions and uniform bounds}\label{sec:existence}

This section is dedicated to the proof of Theorem~\ref{thm:existence_weak_solution}. We start with a precise definition of the notion of weak solution in Section~\ref{sec:weaksol}. Then, we show \emph{a priori} estimates on weak solutions in Section~\ref{sec:apriori}. Finally, we prove the existence of weak solutions to a truncated system of equations in Section~\ref{sec:existenceproof}. Adapting the \emph{a priori} estimates to the truncated system, we deduce that a weak solution to the truncated system is indeed a weak solution to the initial system.

In the following we use the functional space
$$
H^1_D(\Omega)=\{ u \in H^1(\Omega) \mbox{ such that } u\vert_{\Gamma^D}=0\},
$$
where $u\vert_{\Gamma^D}$ denotes the trace of $u$ on $\Gamma^D$. For any function in $H^1_D(\Omega)$, we denote by $K_p$ the Poincaré-Sobolev constant (depending only on $p$ and $\Omega$) of the continuous and compact embedding $H^1_D(\Omega)\subset L^p(\Omega)$, namely
\begin{equation}\label{eq:embed}
\|u\|_{L^p}\leq K_p \|\nabla u\|_{L^2}\,,\quad \text{for all }u\in H^1_D(\Omega),
\end{equation}
 which holds for any $p\in[1,\infty]$ if $d=1$, $p\in[1,\infty)$ if $d=2$ and $p\in[1,2d/(d-2)]$ if $d>2$.
\subsection{Notion of weak solution}\label{sec:weaksol}

 \begin{definition}\label{def:weak-solution} A quadruplet
$(v_\ions, v_\electrons,v_\holes,\psi)\in \mathbb{R}\times(H^1(\Omega)\cap L^{\infty}(\Omega))^3$ is called a weak solution to \eqref{eq:model-dimless} if it satisfies the mass constraint
\begin{equation}\label{boundary:ions}
\int_\Omega \mathcal{F}_\ions \left(v_{\ions}- z_\ions \psi\right)\, \mathrm{d}x \,=\, M_{\ions}\,,
\end{equation}
and there exists $({\widehat v}_{\electrons}, {\widehat v}_{\holes},{\widehat \psi}) \in (H^1_D(\Omega))^3$
such that $ (v_{\electrons}={\widehat v}_{\electrons} + v_{\electrons}^D,v_{\electrons} ={\widehat v}_{\holes} + v_{\holes}^D,\psi = {\widehat \psi} +\psi^D) $
and for all test functions $ (u_{\electrons},  u_{\holes}, \phi)\in (H^1_D(\Omega))^3,$
  \begin{subequations} \label{eq:weak-solution}
\begin{align}
\int_\Omega n_{\alpha} \nabla {v}_{\alpha} \cdot \nabla {{u_{\alpha}}}\,\mathrm{d} x&=\int_\Omega (G-R(n_{\electrons}, n_{\holes})){{u_{\alpha}}}\,\mathrm{d} x,\quad \mbox{ for }\alpha=\electrons,\holes,\label{eq:weak-solution-vnp}\\
\lambda^2\int_\Omega \nabla{\psi} \cdot \nabla   { \phi} \,\mathrm{d} x
&=\int_\Omega \left(n_{\holes}-n_{\electrons}+z_{\ions}\mathcal{F}_\ions \left(v_\ions- z_\ions \psi\right)+C\right){ \phi} \,\mathrm{d} x
\,,\label{eq:weak-solution-psi}
\end{align}
\end{subequations}
where the densities $n_{\electrons},n_{\holes}$ are defined  by \eqref{eq:state-eq-stationary}.
\end{definition}
\begin{lemma}\label{lem:equiv_density_potentials}
Under Assumption~\ref{main_ass}, for any weak solution $(v_\ions, v_\electrons,v_\holes,\psi)$ in the sense of Definition~\ref{def:weak-solution}, the related densities $n_\ions$, $n_\electrons$, $n_\holes$ belong to $H^1(\Omega)\cap L^{\infty}(\Omega)$. Moreover they are bounded from below by positive constants and satisfy \eqref{eq:model-drift-diff} in the weak sense, that is
\begin{subequations}\label{eq:weak-solution-DD}
\begin{align}
 \int_\Omega (z_\alpha n_{\alpha}\nabla \psi + D_\alpha(n_{\alpha}) \nabla n_{\alpha}) \cdot \nabla {{u_{\alpha}}} \,\mathrm{d} x &=\int_\Omega (G-R(n_{\electrons}, n_{\holes})){{u_{\alpha}}} \,\mathrm{d} x,\quad \mbox{ for }\alpha=\electrons,\holes,\label{eq:weak-solution-DD-vnp}\\
\lambda^2\int_\Omega \nabla \psi\cdot \nabla   { \phi} \,\mathrm{d} x &=\int_\Omega (n_{\holes}-n_{\electrons}+z_{\ions}n_{\ions}+C){ \phi} \,\mathrm{d} x,\label{eq:weak-solution-DD-psi}
\end{align}
\end{subequations}
for any test functions $ (u_{\electrons},  u_{\holes}, \phi)\in (H^1_D(\Omega))^3$. Moreover the diffusion coefficients defined in \eqref{def:D} satisfy $D_\alpha(x)\geq 1$ for all $x\in  (0,+\infty)$.
\end{lemma}
\begin{proof} The bounds and $H^1(\Omega)$ regularity of the densities follow from the state equation \eqref{eq:state-eq-stationary} and Assumption~\ref{main_ass}. The drift-diffusion form on the densities then follows from the identity
\[
n_{\alpha} \nabla {v}_{\alpha}= D_\alpha(n_{\alpha}) \nabla n_{\alpha}+z_\alpha n_\alpha \nabla \psi.
\]
Finally the bound $D_\alpha(x)\geq 1$ is a consequence of \eqref{hyp:statistics-n-p}.
\end{proof}

\subsection{A priori estimates}\label{sec:apriori}

In the following, we derive upper bounds for the densities $n_{\electrons}$ and $n_{\holes}$ in Section~\ref{sec:upper}. Then, we derive upper and lower bounds on the potentials $v_{\electrons}, v_{\holes}, v_{\ions}$ and $\psi$ in Section~\ref{sec:boundspot}.

\subsubsection{Upper bounds on the densities}\label{sec:upper}

\begin{proposition}\label{prop:upper-bounds-n}
Under Assumption \ref{main_ass}, there exists an explicit constant ${\overline N}\geq N^D$ such that any weak solution to \eqref{eq:model-dimless} in the sense of Definition \ref{def:weak-solution} verifies
\begin{equation}\label{upper-bounds-n}
0<n_\electrons \leq {\overline N} \mbox{ and }0<n_\holes \leq {\overline N}\mbox{ almost everywhere in } \Omega.
\end{equation}
\end{proposition}

\begin{proof} The lower bounds follow from Lemma~\ref{lem:equiv_density_potentials}. In order to prove the upper bounds, we will adapt a method essentially due to  Stampacchia \cite{Stampacchia_1963, Stampacchia_1965}.
For $\alpha\in \{\electrons,\holes\}$, we define the function ${\mathbb E}_{\alpha} : \nu\in [N^D,+\infty)\mapsto  {\mathbb E}_{\alpha} (\nu)\in {\mathbb R}_+$ by:
\begin{equation}\label{def:E_alpha}
{\mathbb E}_{\alpha}(\nu)=\int_\Omega \left\vert \nabla \log \left(1+ \bigl(\frac{n_\alpha}{\nu}-1\bigr)_+\right)\right\vert^2 \,\mathrm{d} x, \quad \forall \nu\geq N^D\;,\; \alpha\in \{\electrons,\holes\}\,.
\end{equation}
We note that ${\mathbb E}_\alpha$ is a  non-increasing function of $\nu$. Moreover, as $\log(1+ \bigl(\frac{n_\alpha}{\nu}-1\bigr)_+)$ belongs to $H^1_D(\Omega)$ for $\nu \geq N^D$, the Poincar\'e inequality writes
$$
\left\|\log\left(1+ \bigl(\frac{n_\alpha}{\nu}-1\bigr)_+\right)\right\|_{L^2(\Omega)}^2\leq K_2^2 {\mathbb E}_{\alpha}(\nu),\quad  \forall \nu \geq N^D.
$$
Our aim is to establish the existence of ${\overline N}$ such that
${\mathbb E}_{\alpha}({\overline N})=0$ , so that $n_\alpha \leq {\overline N}$ for $\alpha=\electrons,\holes$.

Let us introduce  the non-negative function, bounded by $1$, vanishing on $(-\infty,0)$, defined by
$$
\varphi : s\in{\mathbb R} \mapsto \frac{s}{1+s} {\mathbf 1}_{s\geq 0}.
$$
and the auxiliary function $w_{\nu,\alpha}=\frac{n_\alpha}{\nu}-1$.
They satisfy $\nabla n_\alpha = \nu\nabla w_{\nu,\alpha}$  and
$$
{\mathbb E}_{\alpha}(\nu)=\int_\Omega \left \vert \nabla \log \left(1+ \left(w_{\nu,\alpha}\right)_+\right)\right\vert^2 \,\mathrm{d} x =\int_\Omega \varphi'(w_{\nu,\alpha}) \vert \nabla w_{\nu,\alpha}\vert^2 \,\mathrm{d} x.
$$

For the test function $u_{\alpha}= \frac{\varphi( w_{\nu,\alpha})}{\nu}$, the lower bound on the diffusion coefficient $D_\alpha$ implies that
\begin{equation}\label{energy-meth:ineq1}
\int_\Omega D_\alpha(n_\alpha)\nabla n_\alpha \cdot \nabla u_\alpha \,\mathrm{d} x =\int_\Omega D_\alpha(n_\alpha)\nabla w_{\nu,\alpha}\cdot \nabla \varphi( w_{\nu,\alpha}) \,\mathrm{d} x \geq {\mathbb E}_{\alpha}(\nu), \quad \mbox{ for } \alpha=\electrons,\holes.
\end{equation}
Moreover, the upper bound \eqref{hyp:rate} implies that $R(n_{\electrons}, n_{\holes})\geq -r_0$, so that, with the $L^p$ bound on $G$, we obtain thanks to the Hölder inequality
\begin{multline}
\int_\Omega (G-R(n_{\electrons}, n_{\holes}))(u_{\electrons}+u_{\holes}) \,\mathrm{d} x=\frac{1}{\nu} \int_\Omega (G-R(n_{\electrons},n_{\holes}))\left(\varphi(w_{\nu,\electrons})+\varphi( w_{\nu,\holes})\right) \,\mathrm{d} x\\
\leq \frac{\Vert G +r_0\Vert_{L^p}}{N^D} \left(\left\|{\mathbf 1}_{\{ n_{\electrons}\geq \nu\}}\right\|_{L^{p'}} + \left\|{\mathbf 1}_{\{ n_{\electrons}\geq \nu\}}\right\|_{L^{p'}}\right).
\label{energy-meth:ineq2}
\end{multline}
Then, applying \eqref{eq:weak-solution-DD-vnp} for $\alpha=\electrons, \holes$ with the corresponding $u_{\alpha}$, summing the two equalities and using \eqref{energy-meth:ineq1}, \eqref{energy-meth:ineq2}, we get
\begin{multline}
 {\mathbb E}_{\electrons}(\nu)+ {\mathbb E}_{\holes}(\nu)\leq \int_\Omega \nabla \psi\cdot \left((w_{\nu,\electrons}+1) \nabla \varphi(w_{\nu,\electrons})-(w_{\nu,\holes}+1)\nabla \varphi(w_{\nu,\holes})\right) \,\mathrm{d} x \\
+\frac{\| G +r_0\|_{L^p}}{N^D} \left(\left\|{\mathbf 1}_{\{ n_{\electrons}\geq \nu\}}\right\|_{L^{p'}} + \left\|{\mathbf 1}_{\{ n_{\holes}\geq \nu\}}\right\|_{L^{p'}}\right).\label{energy-meth:ineq3}
\end{multline}
Using in  \eqref{eq:weak-solution-DD-psi} the test function $\phi = \log(1+(w_{\nu,\electrons})_+)-\log (1+(w_{\nu,\holes})_+)$, which satisfies
$$
\nabla \phi=(w_{\nu,\electrons}+1) \nabla \varphi(w_{\nu,\electrons})-(w_{\nu,\holes}+1)\nabla \varphi(w_{\nu,\holes}),
$$
we obtain that
\begin{multline}
\int_\Omega \nabla \psi\cdot \left((w_{\nu,\electrons}+1) \nabla \varphi(w_{\nu,\electrons})-(w_{\nu,\holes}+1)\nabla \varphi(w_{\nu,\holes})\right) \,\mathrm{d} x=\\
\frac{1}{\lambda^2}\int_\Omega (n_{\holes}-n_{\electrons}+
z_{\ions}n_\ions+C) \left(\log\left(1+\bigl(\frac{n_\electrons}{\nu}-1\bigr)_+\right)-\log\left(1+\bigl(\frac{n_\holes}{\nu}-1\bigr)_+\right)\right) \,\mathrm{d} x\\
\leq \frac{1}{\lambda^{2}}\int_\Omega  (z_{\ions}n_\ions+ C)\left(\log\left(1+\bigl(\frac{n_\electrons}{\nu}-1\bigr)_+\right)-\log\left(1+\bigl(\frac{n_\holes}{\nu}-1\bigr)_+\right)\right) \,\mathrm{d} x,\label{energy-meth:ineq4}
\end{multline}
due to the monotonicity of the $\log$ function. Now, since $p>d/2$, one has that $q = 2p'= 2p/(p-1)$ is such that  $H^1_D(\Omega)$ embeds continuously in $L^q(\Omega)$, see \eqref{eq:embed}. Moreover $\frac2q +\frac1p = 1$  so using successively Hölder,  Sobolev and Young's inequalities one obtains
\begin{align*}
&\frac{1}{\lambda^{2}}
\left\vert\int_\Omega\left(C+z_{\ions}n_\ions
\right)\left(\log(1+(w_{\nu,\alpha})_+\right) \,\mathrm{d} x\right\vert\\
&\qquad\leq  \frac{1}{\lambda^{2}}\left\|C+z_\ions n_\ions\right\|_{L^p} \left\|{\mathbf 1}_{\{ n_{\alpha}\geq \nu\}}\right\|_{L^q}\left\|\log(1+(w_{\nu,\alpha})_+)\right\|_{L^q}\\
&\qquad\leq\frac{1}{2}{\mathbb E}_{\alpha}(\nu)
+\frac{K_{q}^2}{2\lambda^4}
\left\|C+z_\ions n_\ions\right\|_{L^p}^2
\left\|{\mathbf 1}_{\{ n_{\electrons}\geq \nu\}}\right\|_{L^q}^2.
\end{align*}
Combined with \eqref{energy-meth:ineq3}-\eqref{energy-meth:ineq4}, this yields, for all $\nu\geq N^D$,
\begin{multline}\label{energy-meth:step1}
{\mathbb E}_{\electrons}(\nu)+ {\mathbb E}_{\holes}(\nu)\\\leq \left( \frac{K_{q}^2 }{\lambda^4}
\left\|C+z_\ions n_\ions\right\|_{L^p}^2
+ \frac{2}{N^D}\| G +r_0\|_{L^p}\right)
\left(\left\|{\mathbf 1}_{\{ n_{\electrons}\geq \nu\}}\right\|_{L^q}^2 + \left\|{\mathbf 1}_{\{ n_{\holes}\geq \nu\}}\right\|_{L^q}^2\right).
\end{multline}
For $r>q = 2p/(p-1)$ such that $H^1_D(\Omega)\subset L^r(\Omega)$, which exists since $p>d/2$ (see \eqref{eq:embed}), observe that for $\mu >\nu\geq N^D$ and  $\alpha= \electrons, \holes$,
$$
{\mathbf 1}_{ \{n_\alpha\geq \mu\}}\leq \left\vert\frac{\log (1+(\frac{n_\alpha}{\nu}-1)_+)}{\log (1+(\frac{\mu}{\nu}-1)_+)}\right\vert^r
{\mathbf 1}_{ \{n_\alpha\geq \nu\}}.
$$
This implies
$$
\int_\Omega {\mathbf 1}_{ \{n_\alpha\geq \mu\}} \,\mathrm{d} x \leq K_{r}^r \frac{({\mathbf E}_{\alpha}(\nu))^{r/2}}{(\log \mu -\log \nu)^r}
$$
and, from \eqref{energy-meth:step1}, we deduce
\begin{equation}\label{energy-meth:step2}
{\mathbb E}_{\electrons}(\mu)+ {\mathbb E}_{\holes}(\mu)\leq \zeta\frac{({\mathbb E}_{\electrons}(\nu)+ {\mathbb E}_{\holes}(\nu))^{r/q}}{(\log \mu -\log\nu)^{2r/q}} \quad \forall \mu>\nu\geq N^D,
\end{equation}
with
$$
\zeta= K_{r}^{2r/q}\left(  \frac{K_{q}^2 }{\lambda^4}
\bigl\||C|+|z_\ions|S_\ions\bigr\|_{L^p}^2
+ \frac{2}{N^D}\bigl\| G +r_0\bigr\|_{L^p}\right).
$$ 
to the function ${\mathcal E}= {\mathbb E}_\electrons \circ \exp + {\mathbb E}_\holes \circ \exp$ defined on $(\log(N^D),+\infty)$, nonnegative and nonincreasing. As it satisfies
$$
{\mathcal E}(y)\leq \zeta \frac{{\mathcal E}(x)^{\beta}}{(y-x)^{\alpha}} \quad \forall y>x\geq \log N^D,
$$
with $\beta=r/q>1$ and $\alpha=2r/q>0$, we obtain the existence of ${\overline N}>N^D$ such that ${\mathcal E}(\log({\overline N}))=0$ and, therefore, ${\mathbb E}_\electrons({\overline N})={\mathbb E}_\holes({\overline N})=0$. Using that from \eqref{energy-meth:step1}, one has $\mathcal{E}(\log(N^D))\leq 2\zeta|\Omega|^{2/q}/K_r^{2r/q}$, the upper bound can be taken explicitly as 
\[
{\overline N} = N^D \exp\left(K \left(\frac{\|C\|_{L^p} + |z_\ions|S_\ions}{\lambda^2} + \frac{\|G\|_{L^p}^{\frac12} + r_0^{\frac12}}{(N^D)^{\frac12}}\right)\right).
\]
with $K$ depending only on $p$ and $\Omega$. This concludes the proof of Proposition~\ref{prop:upper-bounds-n}.

\end{proof}

\subsubsection{Bounds on the potentials}\label{sec:boundspot}

\begin{proposition}\label{prop:bounds-potentials}
Under Assumption \ref{main_ass}, there exists explicit constants ${\overline M}_\psi$ and ${\overline M}_v$ such that any weak solution to \eqref{eq:model-dimless} in the sense of Definition \ref{def:weak-solution} verifies
\begin{subequations}
\begin{align}
-{\overline M}_{\psi}&\leq \psi\leq {\overline M}_{\psi},\label{Linf-bounds-psi}\\
-{\overline M}_v &\leq v_\alpha \leq {\overline M}_v, \mbox{ for } \alpha= \electrons,\holes.\label{Linf-bounds-v}
\end{align}
\end{subequations}
\end{proposition}

\begin{proof}
Due to Proposition~\ref{prop:upper-bounds-n}, the assumptions $C\in L^p(\Omega)$ with $p>d/2$ and $\mathcal{F}_\ions \in L^{\infty}(\mathbb{R})$, we know that the right-hand-side of the Poisson equation is bounded in $L^p(\Omega)$. If $p=\infty$, then the bounds \eqref{Linf-bounds-psi} follow from the maximum principle. If $p$ is merely such that  $p>d/2$ the same estimates as in the previous Section~\ref{sec:upper}, for the much simpler Poisson equation $-\Delta \psi = f\in L^p(\Omega)$ can be replicated to obtain \eqref{Linf-bounds-psi}.

Then, the state equations \eqref{eq:state-eq-stationary} imply that for $\alpha= \electrons,\holes$,
\begin{equation}
v_\alpha \leq {\mathcal F}_\alpha^{-1}({\overline N}) + {\overline M}_\psi,
\end{equation}
and we can set  ${\overline M}_v=\max\{{\mathcal F}_{\electrons}^{-1}({\overline N}) + {\overline M}_\psi,\, {\mathcal F}_{\holes}^{-1}({\overline N}) + {\overline M}_\psi,\, \Vert  v_\electrons^D\Vert_{L^\infty},\,  \Vert  v_\holes^D\Vert_{L^\infty}\}$.

It remains to show that these potentials are bounded from below. Therefore, we write \eqref{eq:weak-solution-vnp}  with  the test functions $(u_\electrons, u_\holes)\in (H^1_D(\Omega))^2$ defined by
$$
u_{\alpha}= (v_\alpha+ {\overline M}_v)_{-}= - (v_\alpha+ {\overline M}_v) {\mathbf 1}_{\{v_{\alpha}+{\overline M}_v\leq 0 \}} \mbox{ for } \alpha = \electrons, \holes.
$$
On the one hand, we have, for $\alpha= \electrons,\holes$,
\begin{equation}\label{eq:lhs-uv}
\int_\Omega n_\alpha \nabla v_\alpha \cdot \nabla u_\alpha \,\mathrm{d} x =-\int_\Omega n_\alpha \vert \nabla u_\alpha \vert^2 \,\mathrm{d} x,
\end{equation}
while, on the other hand,
$$
\int_\Omega (G-R(n_\electrons, n_\holes))u_\alpha \,\mathrm{d} x\geq \int_\Omega \tilde{r}(n_\electrons, n_\holes)(\exp(v_\electrons+v_\holes)-1) (v_\alpha+ {\overline M}_v){\mathbf 1}_{\{v_{\alpha}+ {\overline M}_v\leq 0\}} \,\mathrm{d} x.
$$
But, $\exp(v_\electrons+v_\holes)=\exp(v_\electrons+{\overline M}_v)\exp(v_\holes-{\overline M}_v)\leq \exp(v_\electrons+{\overline M}_v)$. This implies
$$
(\exp(v_\electrons+v_\holes)-1) (v_\electrons+ {\overline M}_v){\mathbf 1}_{\{v_{\electrons}+ {\overline M}_v\leq 0\}}\geq 0.
$$
Therefore, we get, for $\alpha= \electrons,\holes$,
\begin{equation}\label{eq:rhs-uv}
\int_\Omega (G-R(n_\electrons, n_\holes))u_\alpha \,\mathrm{d} x\geq 0
\end{equation}
and from \eqref{eq:lhs-uv} and \eqref{eq:rhs-uv} we deduce
$$
\int_\Omega n_\alpha \vert \nabla u_\alpha \vert^2 \,\mathrm{d} x =0,
$$
which yields, since $n_\alpha>0$ and $u_\alpha\in H^1_D(\Omega)$, that $u_\alpha=0$ and $v_\alpha \geq -{\overline M}_v$ and concludes the proof of Proposition~\ref{prop:bounds-potentials}.
\end{proof}

It remains to prove the bounds on $v_\ions$. For this, we start with a preliminary well-posedness result concerning the nonlinear Poisson equation \eqref{eq:weak-solution-DD-psi} with mass constraint \eqref{boundary:ions}. The proof follows from elementary calculus of variation arguments, which are inspired by \cite{cances2025convergence}. This lemma will also be useful in the next section.
\begin{lemma}\label{ex:uniqueness:nonlinear:Poisson}
Under Assumption~\ref{main_ass}, for any $f\in L^p(\Omega)$ with $p\in[1,\infty]$ and $p>d/2$, and $\sigma\in[0,1]$,  there exists a unique weak solution $(v_\ions, \psi)$, with $v_\ions\in \mathbb{R}$ and $\psi-\sigma\psi^D \in H^1_D(\Omega)$,  to the nonlinear Poisson equation
\[
-\lambda^2\Delta\psi = \sigma(f + z_{\ions}\mathcal{F}_{\ions}(v_\ions - z_\ions \psi))\ \text{in}\ \Omega,\text{ and }
\psi = \sigma\psi^D\ \text{on}\ \Gamma^D,
\]
with the mass constraint
\begin{equation*}
M_\ions = \int_\Omega\mathcal{F}_{\ions}(v_\ions - z_\ions \psi) \,\mathrm{d} x,
\end{equation*}
and it satisfies the estimate
\[
\sigma|v_\ions| + \|\psi\|_{H^1}^2\leq \sigma K(1+\|\psi^D\|_{H^1}^2 + \|f\|_{L^p}^2),
\]
for some explicit constant $K>0$.
\end{lemma}
\begin{proof} For $\hat{\psi}\in H^1_D(\Omega)$ and $v\in\mathbb{R}$, consider the functional
\[
J(\hat{\psi}, v) = \int_\Omega \left[\frac{\lambda^2}{2}|\nabla(\hat{\psi}+\sigma\psi^D)|^2 + \sigma \mathcal{G}_\ions(v-z_\ions(\hat{\psi}+\sigma\psi^D)) - \sigma\left( z_\ions\frac{M_\ions}{|\Omega|}+ f\right)\hat{\psi}\right] \,\mathrm{d} x,
\]
where
\[
\mathcal{G}_\ions(x) = \int_{-\infty}^x\mathcal{F}_\ions(y) \, \mathrm{d} y - \frac{M_\ions}{|\Omega|}x.
\]
Clearly, from \eqref{hyp:statistics-a}, $\mathcal{G}_\ions$ is a strictly convex function and one checks that $J$ is a  well-defined strictly convex functional with values in $\mathbb{R}\cup\{+\infty\}$, whose critical points coincide exactly with weak solutions to the nonlinear Poisson equation with mass constraint. Moreover, thanks to \eqref{hyp:statistics-a} and \eqref{Mass:compatibility}, there are $\mu, K>0$ such that
\[
S_\ions |x|  + K\geq \mathcal{G}_\ions(x)\geq \mu|x|-K \text{ for all }x\in\mathbb{R}.
\]
Therefore, using the lower bound as well as Young's, Hölder and Sobolev inequalities (observe that $H^1_D(\Omega)\subset L^{p'}(\Omega)$ under the hypotheses) one finds that
\[
J(\hat{\psi}, v)\geq \frac{\lambda^2}{4}\|\nabla\hat{\psi}\|^2_{L^2} + \sigma\mu|\Omega| |v| - \sigma K(1+\|f\|_{L^p}^2+\|\nabla\psi^D\|^2_{L^2}),
\]
for some constant $K>0$ depending only on $p$, $\Omega$, $\mathcal{F}_\ions$, $|z_\ions |$ and $M_\ions$. Therefore, $J$ is coercive and admits a unique minimizer which is the unique solution to the nonlinear Poisson equation with mass constraint. The estimate on the solution is obtained by noticing that for the minimizer one has $J(\hat{\psi}, v)\leq J(0,0)$.
\end{proof}

\begin{corollary}\label{cor:bound_va}
Under Assumption \ref{main_ass}, there exists an explicit constant ${\overline M}_\ions$ such that any weak solution to \eqref{eq:model-dimless} in the sense of Definition \ref{def:weak-solution} verifies
\begin{equation}
-{\overline M}_\ions \leq v_\ions \leq {\overline M}_\ions.\label{Linf-bounds-va}
\end{equation}
\end{corollary}

\subsection{Existence proof}\label{sec:existenceproof}
We now prove the existence of weak solutions in the sense of Definition~\ref{def:weak-solution}. The proof is based on two intermediate steps. First, we consider a modified system of equations
based on some truncations of the densities and of the potentials. Then, we obtain the existence of a weak solution to the truncated system  in Proposition~\ref{prop:existence_trunc_system} as a consequence of the Leray-Schauder fixed point theorem. Second, we prove that any weak solution to the truncated system is indeed a weak solution to the initial system, for a sufficiently big level of truncation.

\subsubsection{Existence of a solution to a system involving truncations}

Let us start with the definition of different truncations. For $M>0$, we define $T_M:{\mathbb R}\to [M^{-1},M]$ and $S_M:{\mathbb R}\to [-M,M]$ as
$$
 T_M(x)=\left\{\begin{array}{ccl} \frac{1}{M} &\mbox{if}& x\leq  \frac{1}{M},\\[2.mm]
x &\mbox{if} &\frac{1}{M}\leq x\leq M,\\[2.5mm]
M &\mbox{if}& x\geq M,
\end{array}
\right.
\mbox{ and }
S_M(x)=\left\{\begin{array}{ccl} -M &\mbox{if}& x\leq  -M,\\[2.mm]
x& \mbox{if}& -M\leq x\leq M,\\[2.mm]
M &\mbox{if}& x\geq M.
\end{array}
\right.
$$
Then, we define the truncation of the diffusion coefficient function $D_\alpha$ as $D_{\alpha,M} :  {\mathbb R}\to {\mathbb R}$ satisfying
$$
D_{\alpha,M} (x)= D_\alpha(T_M(x)) \quad \forall x\in {\mathbb R}.
$$
A direct consequence of the lower bound of $D_\alpha$ is that $D_{\alpha,M}(x)\geq 1$ for all $x\in {\mathbb R}$. Let us now introduce the $C^1$-diffeomorphism ${\mathcal F}_{\alpha,M} : {\mathbb R}\to {\mathbb R}$ defined through its reciprocal function by
\[
{\mathcal F}_{\alpha,M}^{-1}(x) = \int_1^x({\mathcal F}_{\alpha}^{-1})' (T_M(y)) \, \mathrm{d} y \quad \forall x\in {\mathbb R}.
\]
This function is such that
\begin{equation}\label{prop_DM}
D_{\alpha,M}(x)=T_M(x) ({\mathcal F}_{\alpha,M}^{-1})' (x)  \quad \forall x\in {\mathbb R}.
\end{equation}
We finally define the truncation of the recombination-generation term $R$ by
\begin{equation}\label{def_RM}
R_M(x,y,z)= \tilde{r}(T_M(x), T_M(y))\left(e^{ S_M({\mathcal F}_{\electrons,M}^{-1}(x)-z)}e^{S_M({\mathcal F}_{\holes,M}^{-1}(y)+z)}-1\right)\quad  \forall (x,y,z)\in {\mathbb R}^3.
\end{equation}
Observe that because of the truncation $S_M$ we introduce a third argument $z$ dedicated to the electric field.

\begin{proposition}\label{prop:existence_trunc_system}
Under Assumption~\ref{main_ass}, for all $M>0$,  there exists $({\psi}^M,{ n}_{\electrons}^M,{n}_{\holes}^M, v_\ions^M)$ such that
\[(\psi^M-\psi^D, n_\electrons^M- n_\electrons^D, n_\holes^M- n_\holes^D)\in (H^1_D(\Omega))^3\ \text{ and }\ v_\ions^M\in\mathbb{R},\]
the mass constraint \eqref{boundary:ions} is satisfied by $(v_{\ions}^M, \psi^M)$, that is
\begin{equation*}
\int_\Omega \mathcal{F}_\ions \left(v_{\ions}^M- z_\ions \psi^M\right) \,\mathrm{d} x \,=\, M_{\ions}\,,
\end{equation*}
and, for all $({\phi},{u}_\electrons,{u}_\holes)\in (H^1_D(\Omega))^3$ and $\alpha\in\{\electrons,\holes\}$, $({\psi}^M,{ n}_{\electrons}^M,{n}_{\holes}^M, v_\ions^M)$ satisfies the weak formulation of the truncated system
\begin{subequations}\label{def:trunc_weak_solution}
\begin{align}
&\lambda^2\int_\Omega\nabla \psi^M\cdot \nabla \phi \,\mathrm{d} x = \int_\Omega \left(T_M(n_\holes^M)-T_M(n_\electrons^M)  +\, z_{\ions}\,\mathcal{F}_\ions \left(v_{\ions}^M- z_\ions \psi^M\right) +C\right)\phi \,\mathrm{d} x, \label{def:trunc_weak_solution_psi}\\
& \int_\Omega (D_{\alpha,M}(n_\alpha^M)\nabla n_\alpha^M+z_\alpha\,T_M(n_\alpha^M)\nabla \psi^M)\cdot \nabla u_\alpha \,\mathrm{d} x =\int_{\Omega} (G-R_M(n_\electrons^M,n_\holes^M,\psi^M))u_\alpha \,\mathrm{d} x.\label{def:trunc_weak_solution_nalpha}
\end{align}
\end{subequations}

\end{proposition}
\begin{proof}
In order to prove Proposition \ref{prop:existence_trunc_system}, we introduce an application $\tau : L^2(\Omega)^2\times [0,1]\to L^2(\Omega)^2$  such that $\tau ({\bar n}_\electrons,{\bar n}_\holes, \sigma)= (n_\electrons, n_\holes)$. This application is defined in two steps.

{\em Step 1}:  First, thanks to Lemma~\ref{ex:uniqueness:nonlinear:Poisson}, for a given $({\bar n}_\electrons,{\bar n}_\holes, \sigma)$, there is a unique function $\psi={\widehat \psi}+\sigma \psi^D \in H^1(\Omega)$ such that $\hat{\psi}\in H^1_D(\Omega)$ and a unique $v_\ions\in\mathbb{R}$ solutions to the following problem:  for all $\phi\in H^1_D(\Omega)$,
\begin{equation}\label{weak_form_psihat}
\lambda^2\int_\Omega\nabla \psi\cdot \nabla \phi \,\mathrm{d} x= \sigma \int_\Omega \left(T_M(\bar{n}_\holes)-T_M(\bar{n}_\electrons)  +\, z_{\ions}\,\mathcal{F}_\ions \left(v_{\ions}- z_\ions \psi\right) +C\right)\phi \,\mathrm{d} x,
\end{equation}
with the mass constraint
\begin{equation}\label{boundary:ions:lemma}
\int_\Omega \mathcal{F}_\ions \left(v_{\ions}- z_\ions \psi\right) \,\mathrm{d} x \,=\, M_{\ions}\,.
\end{equation}
Furthermore, there is $K$ independent of $\sigma\in [0,1]$ , such that
\begin{equation}\label{H1bound_psi}
\|\nabla {\psi}\|_{L^2} \leq K.
\end{equation}

{\em Step 2}: Applying the Lax-Milgram theorem, we obtain the existence and uniqueness of ${\widehat n}_\electrons \in H^1_D(\Omega)$ and ${\widehat n}_\holes \in H^1_D(\Omega)$ such that, defining $n_\electrons ={\widehat n}_\electrons + \sigma n^D_\electrons$ and $n_\holes ={\widehat n}_\holes + \sigma n^D_\holes$, one has for all $u_\electrons, u_\holes \in H^1_D(\Omega)$ that
\begin{multline}\label{weak_form_n}
\int_\Omega (D_{\electrons,M}({\bar n}_\electrons) \nabla {n}_\electrons \cdot\nabla u_\electrons + D_{\holes,M}({\bar n}_\holes) \nabla {n}_\holes \cdot\nabla u_\holes) \,\mathrm{d} x=\\
+\int_\Omega \nabla \psi \cdot \left( T_M({\bar n}_\electrons) \nabla u_\electrons-T_M({\bar n}_\holes) \nabla u_\holes\right) \,\mathrm{d} x
+\sigma \int_\Omega (G-R_M({\bar n}_\electrons,{\bar n}_\holes,\psi))(u_\electrons+u_\holes) \,\mathrm{d} x.
\end{multline}
 Using $u_\electrons={\widehat n}_\electrons$ and $u_\holes={\widehat n}_\holes$ as test functions in \eqref{weak_form_n} and once again classical inequalities (observe for the generation term that, again, $H^1_D(\Omega)\subset L^{p'}(\Omega)$ under the hypotheses), we obtain the existence of $K$ which depends (in particular) on $M$ but not on $\sigma\in [0,1]$ such that
\begin{equation}\label{H1bound_n}
\Vert\nabla {n}_\electrons \Vert_{L^2}^2 + \Vert \nabla {n}_\holes \Vert_{L^2}^2  \leq K \left( \|\nabla {\psi}\|_{L^2}^2 +1\right).
\end{equation}

The mapping $\tau :  ({\bar n}_\electrons,{\bar n}_\holes, \sigma)\in  L^2(\Omega)^2\times [0,1]\to  (n_\electrons, n_\holes) \in L^2(\Omega)^2$ is well-defined through Step~1 and Step~2. It satisfies $\tau({\bar n}_\electrons,{\bar n}_\holes, 0)= 0$.
Moreover, we deduce from \eqref{H1bound_n}, combined with \eqref{H1bound_psi}, that $\tau$ is a compact mapping from $L^2(\Omega)^2\times [0,1]$ to $L^2(\Omega)^2$ and that there exists a constant $K$ such that $\Vert n_\electrons\Vert_{L^2}, \Vert n_\holes\Vert_{L^2}\leq K$ for all $(n_\electrons,n_\holes,\sigma)\in L^2(\Omega)^2\times [0,1]$ satisfying $\tau (n_\electrons,n_\holes,\sigma)=(n_\electrons,n_\holes)$.
Therefore, the Leray-Schauder fixed point theorem (see Theorem 11.6 in \cite{Gilbarg-Trudinger}) yields the existence of a fixed point
$(n_\electrons^M,n_\holes^M)\in (L^2(\Omega))^2$ to $\tau (\cdot,\cdot,1)$. Moreover, ${\psi}^M\in H^1(\Omega)$ is  obtained as a by-product of Step 1: ${\psi}^M={\widehat \psi} +\psi^D$  where ${\widehat \psi}\in H^1_D(\Omega)$ and $v_\ions^M\in\mathbb{R}$ are defined by \eqref{weak_form_psihat}-\eqref{boundary:ions:lemma} with $\sigma =1$ and $ (n_\electrons^M,n_\holes^M)$ instead of ${\bar n}_\electrons, {\bar n}_\holes$ in the right-hand-side.
The densities $n_\electrons^M$, $n_\holes^M$ belong to $H^1(\Omega)$ and  ${\widehat n}_\electrons=n_\electrons^M-n_\electrons^D$, ${\widehat n}_\holes = n_\holes^M -n_\holes^D$ belong to $H^1_D(\Omega)$. Noticing that $(\psi^M, n_\electrons^M, n_\holes^M)$ satisfies \eqref{def:trunc_weak_solution}, this ends  the proof of Proposition \ref{prop:existence_trunc_system}.
\end{proof}

Knowing $({\psi}^M,{ n}_\electrons^M,{n}_\holes^M,v_\ions^M)$ is a solution to the truncated system  \eqref{def:trunc_weak_solution}, we can define some associate quasi-Fermi potentials as
\begin{equation}\label{def_trunc_QFpot}
v_\electrons^M={\mathcal F}_{\electrons,M}^{-1}(n_\electrons^M)-\psi^M,\quad  v_\holes^M={\mathcal F}_{\holes,M}^{-1}(n_\holes^M)+\psi^M.
\end{equation}
We have that $(v_\electrons^M, v_\holes^M) \in H^1(\Omega)^2$,
$ (v_\electrons^M- ( {\mathcal F}_{\electrons,M}^{-1}(n_\electrons ^D)-\psi^D), v_\holes^M- ( {\mathcal F}_{\holes,M}^{-1}(n_\holes ^D)+\psi^D))\in H^1_D(\Omega)^2$ and for all $(u_\electrons, u_\holes)\in (H^1_D(\Omega))^2$,
\begin{subequations}\label{trunc_weak_solution_QFpot}
\begin{align}
& \int_\Omega (T_M(n_\alpha^M)\nabla v_\alpha^M)\cdot \nabla u_\alpha \,\mathrm{d} x =\int_{\Omega} (G-R_M(n_\electrons^M,n_\holes^M,\psi^M))u_\alpha \,\mathrm{d} x, \mbox{ for } \alpha=\electrons,\holes.\label{def:trunc_weak_solution_QFalpha}
\end{align}
\end{subequations}

\subsubsection{Proof of Theorem \ref{thm:existence_weak_solution}}

The \emph{a priori} bounds on weak solutions follow from Proposition \ref{prop:upper-bounds-n}, Proposition \ref{prop:bounds-potentials} and Corollary~\ref{cor:bound_va}. Observe that the bounds on the potentials also yield uniform lower bounds on the densities.
It remains to show the existence part.
To proceed, we will now establish $L^{\infty}$ bounds on the solutions to the truncated system $(\psi^M, n_\electrons^M, n_\holes^M, v_\ions^M)$ defined by \eqref{def:trunc_weak_solution} and to the associated quasi-Fermi potential $(v_\electrons^M, v_\holes^M)$ defined by \eqref{def_trunc_QFpot}.
The crucial point consists in proving that these $L^{\infty}$ bounds do not depend on $M$ for $M$ sufficiently large, so that a solution to the truncated system is indeed a solution to the initial system. In order to obtain the $L^{\infty}$-estimates for the solutions to the truncated system, we follow the main lines of Section \ref{sec:apriori}.

{\em Step 1.} Upper bound on the densities

In order to establish an upper bound on the densities, we follow the lines of the Proof of Proposition \ref{prop:upper-bounds-n}.
The functions  ${\mathbb E}_{\alpha}^M: [N^D,+\infty)\to{\mathbb R}_+$ are defined by substituting  $n_\alpha^M$ to  $n_\alpha$ in the definition \eqref{def:E_alpha} of ${\mathbb E}_\alpha$. The auxiliary functions are now denoted by $w_{\nu,\alpha}^M=n_\alpha^M/\nu-1$.

For $\nu\geq N^D$, we follow the computations leading to \eqref{energy-meth:ineq3}, which yields
\begin{multline}
 {\mathbb E}_{\electrons}^M(\nu)+ {\mathbb E}_{\holes}^M(\nu)\leq \int_\Omega \nabla \psi\cdot \left(\frac{T_M(\nu(w_{\nu,\electrons}^M+1))}{\nu} \nabla \varphi(w_{\electrons}^M)-\frac{T_M(\nu(w_{\nu,\holes}^M+1)}{\nu}\nabla \varphi(w_{\nu,\holes}^M)\right) \,\mathrm{d} x \\
+\frac{\| G +r_0\|_{L^p}}{N^D} \left(\left\|{\mathbf 1}_{\{ n_{\electrons}^M\geq \nu\}}\right\|_{L^{p'}} + \left\|{\mathbf 1}_{\{ n_{\holes}^M\geq \nu\}}\right\|_{L^{p'}}\right).\label{energy-meth:ineq3-M}
\end{multline}
Now, we introduce the function $h_\nu^{M}$ defined by
$$
h_\nu^{M} (s)=\int_0^s \frac{T_M(\nu(w+1))}{\nu}\varphi'(w) \, {\rm d}w \quad \forall s\in{\mathbb R}.
$$
It is a non-decreasing and non-negative function that vanishes on $(-\infty,0]$, so that we can obtain
the inequality equivalent to \eqref{energy-meth:ineq4}:
\begin{multline}
 {\mathbb E}_{\electrons}^M(\nu)+ {\mathbb E}_{\holes}^M(\nu)
\leq \frac{\| G +r_0\|_{L^p}}{N^D} \left(\left\|{\mathbf 1}_{\{ n_{\electrons}^M\geq \nu\}}\right\|_{L^{p'}} + \left\|{\mathbf 1}_{\{ n_{\holes}^M\geq \nu\}}\right\|_{L^{p'}}\right) \\
+\frac{1}{\lambda^{2}}\int_\Omega \left( C
+z_{\ions}\,\mathcal{F}_\ions \left(v_{\ions}^M- z_\ions \psi^M\right)
\right)(h_\nu^{M}(w_{\nu,\electrons}^M)-h_\nu^{M}(w_{\nu,\holes}^M)) \,\mathrm{d} x.\label{energy-meth:ineq4-M}
\end{multline}
But, we also observe that if $M\geq \nu^{-1}$, then for all $w\geq 0$, $\nu(w+1)\geq M^{-1}$ which yields that
$$
\frac{T_M(\nu(w+1))}{\nu}\leq (w+1),\ \text{for all } w\geq 0.
$$
This implies that $h_\nu^{M}(s)\leq \log (1+s_+)$ for all $s\in{\mathbb R}$. Therefore, we obtain an inequality similar to \eqref{energy-meth:step1}: for $\nu\geq N^D$ and $M\geq 1/\nu$,
\begin{align*}
{\mathbb E}_{\electrons}^M(\nu)+& {\mathbb E}_{\holes}^M(\nu)\leq \\
&\left( \frac{K_{q}^2 }{\lambda^4}
\left\||C|+|z_\ions|S_\ions\right\|_{L^p}^2
+ \frac{2}{N^D}\| G +r_0\|_{L^p}\right)
\left(\left\|{\mathbf 1}_{\{ n_{\electrons}^M\geq \nu\}}\right\|_{L^{p'}} + \left\|{\mathbf 1}_{\{ n_{\holes}^M\geq \nu\}}\right\|_{L^{p'}}\right).
\end{align*}
Finally, if $M\geq \frac{1}{\overline N}$, with ${\overline N}$ defined in Proposition~\ref{prop:upper-bounds-n}, we can follow the end of the proof of this proposition, so that we obtain
\begin{equation}\label{upper-bounds-n-M}
0<n_\electrons^M, n_\holes^M\leq {\overline N} \mbox { almost everywhere in } \Omega.
\end{equation}

{\em Step 2.} Bounds on the potentials

A direct consequence of \eqref{upper-bounds-n-M} is that the bounds on $\psi^M$ are the same as the bounds on $\psi$, if $M\geq \frac{1}{\overline N}$:
$$
-{\overline M}_\psi\leq \psi^M\leq {\overline M}_\psi.
$$
Let us now choose $M$ that satisfies also $M\geq {\overline N}$, so that ${\mathcal F}_{\alpha,M}^{-1}(n^M_\alpha)= {\mathcal F}_\alpha^{-1}(n^M_\alpha)$ for
$\alpha=\electrons,\holes$. Then, we deduce that
$$
v_\alpha^M\leq {\overline M}_v, \mbox{ for }\alpha =\electrons, \holes,
$$
with ${\overline M}_v$ defined in the proof of Proposition \ref{prop:bounds-potentials}.

Next, we prove that the quasi-Fermi potentials $v_\alpha^M$ are also bounded from below. Using  $u_\electrons=(v_\electrons^M+{\overline M}_v)_-$ as a test function in  \eqref{def:trunc_weak_solution_QFalpha} for $\alpha=\electrons$, we have that
\begin{align*}
-\int_\Omega T_M(n_\electrons^M)\vert \nabla u_\electrons\vert^2 \,\mathrm{d} x&=\int_\Omega (G-R_M(n_\electrons^M,n_\holes^M,\psi^M))u_{\electrons}\\
&\geq\int_\Omega \tilde{r}(T_M(n_\electrons^M,n_\holes^M))(e^{S_M(v_\electrons^M)+S_M(v_\holes^M)} -1)(v_\electrons^M+{\overline M}_v){\mathbf 1}_{v_\electrons^M+{\overline M}_v\leq 0} \,\mathrm{d} x.
\end{align*}
But as soon as $M\geq {\overline M}_v$, we have:
\begin{align*}
    \exp(S_M(v_\holes^M)-{\overline M}_v) &\leq 1,\\
    \exp(S_M(v_\electrons^M)+{\overline M}_v) &\leq  1 \mbox{ if } v_\electrons^M+{\overline M}_v\leq 0.
\end{align*}
 We deduce, as in the end of the proof of Proposition \ref{prop:bounds-potentials} that $u_\electrons=0$ and, therefore, that $v_\electrons^M\geq {-{\overline M}_v}$. The proof is similar for $v_\holes^M$.

 {\em Conclusion.} We have obtained that for $M$ satisfying
 $$
 M\geq \max ({\overline N}^{-1}, {\overline N},  {\overline M}_v),
 $$
 with ${\overline N}$ and $ {\overline M}_v$ respectively defined in Proposition \ref{prop:upper-bounds-n} and Proposition \ref{prop:bounds-potentials}, the solution to the truncated system $(\psi^M,n_\electrons^M, n_\holes^M)$ satisfies bounds that do not depend on $M$.
 This means that the solution to the truncated system defined by \eqref{def:trunc_weak_solution} is a solution to the initial system \eqref{eq:model-dimless}.

\begin{remark}[Multi-layer devices] \label{rem:intrinsic} The model \eqref{eq:model-dimless} is not general enough to describe some typical perovskite solar cell device.
Indeed, generally, one has a multi-layer domain with mobile anion species only in some layers $\Omega_{\text{PVK}}\subset \Omega$ (see Section~\ref{sec:Sim-PSC} for an example). In order to include this type of devices into the modelling framework, for the stationary problem at hand here, it amounts to introduce some characteristic function in the Poisson equation 
\[
- \lambda^2\Delta\psi = n_{\holes} - n_{\electrons} + z_{\ions}\mathcal{F}_\ions \left(v_\ions- z_\ions \psi\right)\mathbf{1}_{\Omega_{\text{PVK}}}(x) + C.
\]
We leave the reader to check that the analysis can be extended easily to this type (and more general) settings.
\end{remark}

\subsubsection{Proof of Corollary \ref{coro:bound:current}}\label{proof:cor}
To prove Corollary \ref{coro:bound:current}, we consider a solution $(v_\ions, v_\electrons,v_\holes,\psi)$ to \eqref{eq:model-dimless} and fix some $B\subset \partial\Omega$ which satisfies $\overline{B} \cap \overline{B^c\cap \Gamma^D}=\emptyset$. We first prove that $I_B$ in \eqref{outward:current} is well defined. To do so, we point out that $j_\alpha$, $\alpha\in\{\electrons,\holes\}$, belongs to $L^2(\Omega)$ by taking the $L^2$ norm in \eqref{outward:current} and applying
Theorem \ref{thm:existence_weak_solution}, which guarantees that $(n_\electrons,n_\holes)$ (resp. $(v_\electrons,v_\holes)$) is explicitly bounded in $L^\infty(\Omega)$ (resp. $H^1(\Omega)$). Furthermore, $\nabla\cdot (j_\electrons+j_\holes)$ also belongs to $L^2(\Omega)$, since the difference between \eqref{eq:model-cont-eq-n-dimless} and \eqref{eq:model-cont-eq-p-dimless} reads $\nabla\cdot (j_\electrons+j_\holes) = 0$. Therefore the normal trace $\nabla\cdot (j_\electrons+j_\holes)\cdot\nu$ is bounded in $H^{-1/2}(\partial \Omega)$ and we have \cite{Girault_Raviart86}
 \[
 \int_{\partial \Omega}(j_\electrons + j_\holes)\cdot \nu \,u\,\mathrm{d}\Gamma\,=\,
 \int_{\Omega}(j_\electrons + j_\holes)\cdot\nabla u\,\mathrm{d}x
 \,,
 \]
 for all $u \in \mathcal{C}^\infty(\bar{\Omega})$. In particular, this guarantees that
 \[
 \int_{\partial \Omega}(j_\electrons + j_\holes)\cdot \nu \,(u_1 - u_2)\,\mathrm{d}\Gamma\,=\,
 \int_{\Omega}(j_\electrons + j_\holes)\cdot\nabla (u_1-u_2)\,\mathrm{d}x
 \,,
 \]
 for all $(u_1,u_2)\in X^2$, where
 \[X\,=\, \{u\in \mathcal{C}^\infty(\bar{\Omega})\,|\,u(x)=1\,,\quad\forall  x\in B\quad \textrm{and}\quad u(x)=0 \,,\quad\forall  x\in B^c\cap \Gamma^D\}\,.\]
 We check that $u_1-u_2\in H^1_D(\Omega)$, which allows us to substitute $j_\electrons+j_\holes$ in the previous right hand side according to \eqref{eq:weak-solution-vnp} with $\alpha\in\{\electrons,\holes\}$, leading to 
 \[
 \int_{\partial \Omega}(j_\electrons + j_\holes)\cdot \nu \,(u_1 - u_2)\,\mathrm{d}\Gamma\,=\,0
 \,.
 \]
 The previous relation justifies that $I_B$ in \eqref{outward:current} is well defined as
 \[
 I_B\,=\,
 \int_{\partial \Omega}(j_\electrons + j_\holes)\cdot \nu \,u\,\mathrm{d}\Gamma
 \,=\,
 \int_{\Omega}(j_\electrons + j_\holes)\cdot\nabla u\,\mathrm{d}x
 \,,
 \]
 for all $u\in X$. Furthermore, taking the absolute value and applying Cauchy-Schwarz inequality in the right hand side, we find
 \[
 |I_B|\,\leq\,(\|j_\electrons\|_{L^2(\Omega)}+ \|j_\holes\|_{L^2(\Omega)})\,\inf_{u\in X}
 \|\nabla u\|_{L^2(\Omega)}
 \,.
 \]
 We conclude the proof by checking that $X$ is nonempty since $\overline{B} \cap (\overline{B^c\cap \Gamma^D})=\emptyset$.

\section{Numerical simulations}\label{sec:Sim}
In this section, we simulate the charge transport in two representative applications to study how the generation rate influences carrier densities and potential profiles.
We employ a stationary version of the time-implicit two-point flux finite volume scheme introduced in \cite{Abdel2023Existence}.
The simulations are performed with the Julia package \href{https://github.com/WIAS-PDELib/ChargeTransport.jl}{\texttt{ChargeTransport.jl}}\cite{ChargeTransport}, which in turn is built on \texttt{VoronoiFVM.jl} \cite{VoronoiFVM}, which implements the Voronoi finite volume method for general nonlinear diffusion-convection-reaction systems.

For both applications, we use the Fermi-Dirac integral of order $1/2$, as stated in \eqref{eq:FD12}, for electrons and holes as statistics function.
We combine radiative and SRH contributions, entering the recombination-generation term in \eqref{eq:recombination-standard},
$
r(n_\electrons, n_\holes) = r_{0, \text{rad}} + (n_\electrons + n_\holes)^{-1},
$
with $r_{0, \text{rad}} = 1$ in the solar cell case (\Cref{sec:Sim-PSC}) and $r_{0, \text{rad}} = 0$ for the laser-beam application (\Cref{sec:Sim-LBIC}).
Both numerical examples are integrated in \cite{ChargeTransport}, where the solar cell simulations are contained in the example file \href{https://github.com/WIAS-PDELib/ChargeTransport.jl/blob/master/examples/Ex109_PSC_NonDimensional.jl}{\texttt{Ex109\_PSC\_NonDimensional.jl}} while the LBIC simulations are conducted with \href{https://github.com/WIAS-PDELib/ChargeTransport.jl/blob/master/examples/Ex203_LBIC_NonDimensional.jl}{\texttt{Ex203\_LBIC\_NonDimensional.jl}}.

\subsection{Three-layer perovskite solar cell} \label{sec:Sim-PSC}
In this example, we simulate the charge transport in a perovskite solar cell in one dimension, consisting of an absorber layer, which captures photons from sunlight, sandwiched between two transport layers, as illustrated in \Cref{fig:PSC-Setup}a.
We partition the domain $\Omega$ into three subdomains $\Omega = \Omega_{\text{ETL}} \cup \Omega_{\text{PVK}} \cup \Omega_{\text{HTL}}$ with $\Omega_{\text{ETL}} = [0, 1]$, $\Omega_{\text{PVK}} = [1, 5]$, and $\Omega_{\text{HTL}} = [5, 7]$.
While we assume in the first setup that $z_\ions = 0$, resulting in a bipolar model, we assume in the second setup $z_\ions = 1$.
In that case, ionic movement is restricted to the middle layer $\Omega_{\text{PVK}}$.
In this one-dimensional set-up, we solely assume outer Dirichlet boundaries.
At the inner interfaces, we assume continuity of potentials as well as their gradients.
With these additional interfacial conditions the above analysis of the model can be extended to multi-layer devices as explained in Remark~\ref{rem:intrinsic}.
Finally, we choose a uniform grid spacing $h \approx 1.26 \times 10^{-2}$, resulting in a total of $558$ nodes.

\begin{figure}[!ht]
    \includegraphics[width=\textwidth]{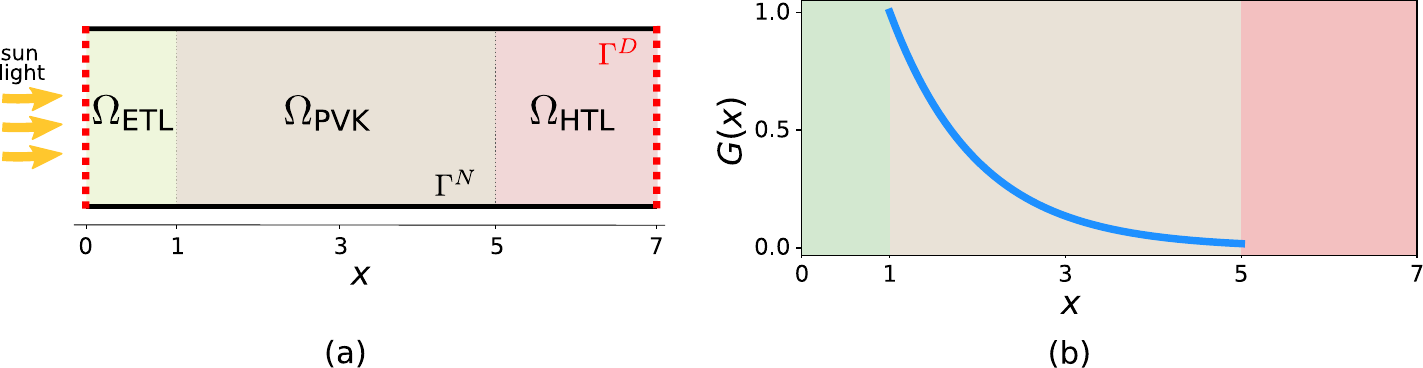}
    \caption{
        (a) Schematic of the three-layer solar cell architecture, with the perovskite layer $\Omega_{\text{PVK}}$ sandwiched between the electron transport layer $\Omega_{\text{ETL}}$ and the hole transport layer $\Omega_{\text{HTL}}$.
        Light enters through the left contact.
        The boundaries $\Gamma^N$ and $\Gamma^D$ are also indicated.
        (b) Exponentially decaying photogeneration rate $G$, present only within the absorbing perovskite layer $\Omega_{\text{PVK}}$.
    }
    \label{fig:PSC-Setup}
\end{figure}

The Debye length is set to $\lambda = 1$.
The doping density is piecewise constant across the three regions with $C = 10$ in $\Omega_{\text{ETL}}$, $C = -10$ in $\Omega_{\text{HTL}}$ and to $C = 0$ in $\Omega_{\text{PVK}}$.
The photogeneration rate, illustrated in \Cref{fig:PSC-Setup}b, is modeled by an exponentially decaying function $G(x) = {\mathbf 1}_{\Omega_{\text{PVK}}} (x) G_0 \exp(-(x-x_1)) $, where $x_1 = 1$ marks the left edge of the absorber layer and $G_0$ is a scaling factor.
Photogeneration is confined to $\Omega_{\text{PVK}}$ with a typical value $G_0 = 1$, corresponding to standard solar illumination conditions, i.e., a clear sunny day.
Moreover, we impose Dirichlet boundary conditions at the contacts:
$v_\electrons^D|_{x = 0} = v_\holes^D|_{x = 0} = 0$, $\psi^D|_{x = 0} = \psi_0$ at the left contact and $-v_\electrons^D|_{x = 7} = v_\holes^D|_{x = 7} = \overline{V}$, $\psi^D|_{x = 7} = \psi_0 + \overline{V}$ at the right contact,
where $\overline{V}$ is the applied voltage and $\psi_0$ the equilibrium solution of the Poisson equation without left-hand side.
This equilibrium solution also serves as the initial guess for the nonlinear solver.
In the following, we examine the impact of the photogeneration rate on charge transport by varying the generation prefactor $G_0$.

\noindent
\textbf{Two-species model.}

\noindent
First, we consider the case of the two species model, i.e., $z_\ions = 0$, with an applied voltage of $\overline{V} = 2$.
\Cref{fig:PSC-Potential-V-2p0} shows the potential profiles for different $G_0$, where brighter colors indicate larger $G_0$ and arrows mark the direction of increase.
\Cref{fig:PSC-Density-V-2p0} displays the corresponding electron and hole densities, along with the $L^\infty$ norms of carrier densities and potentials as functions of $G_0$.

\begin{figure}[!ht]
    \includegraphics[width=\textwidth]{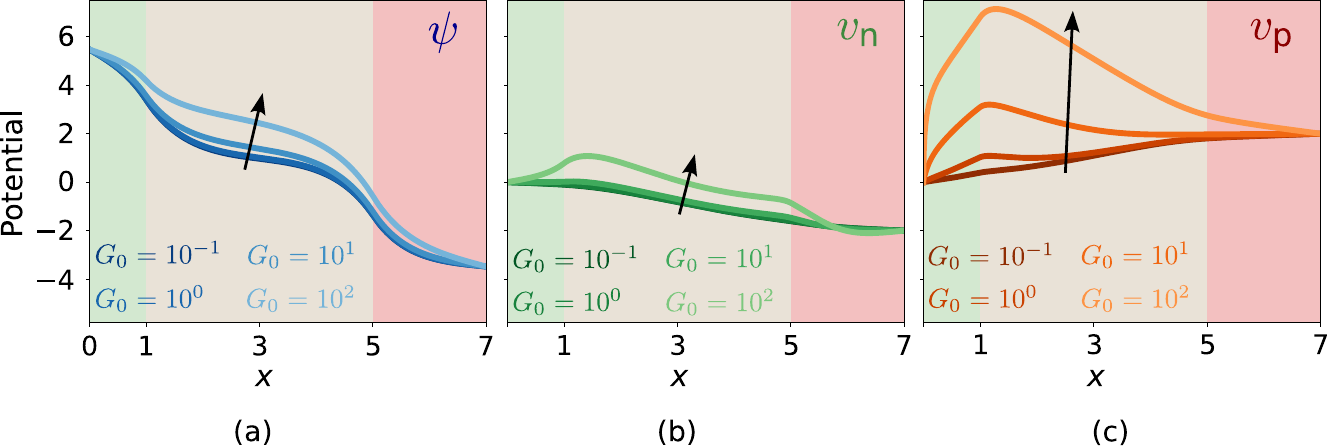}
    \caption{
        (a) Electric potential $\psi$ (blue), and quasi-Fermi potentials of (b) electrons (green) and (c) holes (red), shown for different values of the photogeneration prefactor $G_0$.
        Brighter colors correspond to larger values of $G_0$ with arrows indicating the direction of increasing $G_0$.
        The externally applied voltage, which enters through the Dirichlet boundary conditions, is set to $\overline{V} = 2$.
    }
    \label{fig:PSC-Potential-V-2p0}
\end{figure}

The strongest impact of increasing $G_0$ on the potential profiles appears in the hole potential $v_\text{p}$ (\Cref{fig:PSC-Potential-V-2p0}c).
For increasing $G_0$, $v_\text{p}$ rises sharply at the interface between the electron transport layer (green shaded) and the perovskite layer (gray shaded).
The electron quasi-Fermi potential $v_\text{n}$ (\Cref{fig:PSC-Potential-V-2p0}b) and the electric potential $\psi$ (\Cref{fig:PSC-Potential-V-2p0}a) are also influenced, though markedly only for larger prefactors such as $G_0 = 10^2$.

Since both transport layers have the same absolute doping, the electron $n_\electrons$ and hole $n_\holes$ densities in \Cref{fig:PSC-Density-V-2p0} are uniformly distributed along the transport direction for small $G_0$.
The carrier densities grow in the intrinsic perovskite layer with increasing $G_0$.
However, the hole density increases slightly more than the electron density, consistent with the distribution of the potential profiles.

\begin{figure}[!ht]
    \includegraphics[width=\textwidth]{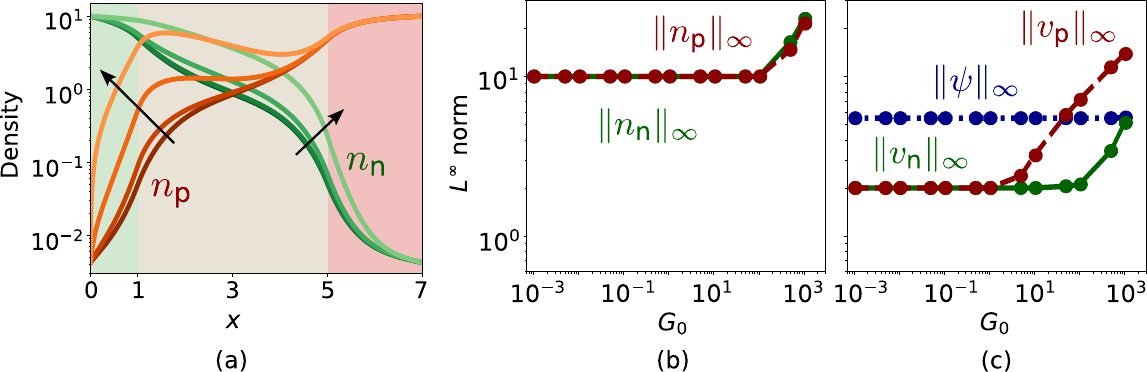}
    \caption{
        (a) Electron density (green) and hole density (red) shown for different values of the photogeneration prefactor $G_0$.
        Brighter colors correspond to larger values of $G_0$ with arrows indicating the direction of increasing $G_0$.
        (b) $L^{\infty}$ norms of the carrier densities of electrons (green), and holes (red dashed) as well as
        (c) $L^{\infty}$ norms of the quasi-Fermi potentials of electrons (green) and holes (red dashed), and the electric potential (blue dotted), plotted as a function of the photogeneration prefactor $G_{0}$.
        The externally applied voltage, which enters through the Dirichlet boundary conditions, is set to $\overline{V} = 2$.
    }
    \label{fig:PSC-Density-V-2p0}
\end{figure}

Examining the $L^{\infty}$ norms in \Cref{fig:PSC-Density-V-2p0}b and \Cref{fig:PSC-Density-V-2p0}c, we find that for smaller choices of the photogeneration prefactor $G_0 < 10^{1}$ that the potentials and carrier densities remain bounded relative to their Dirichlet values.
This observation is consistent with physical expectations.
For larger $G_0$, the $L^{\infty}$ norms of the quasi-Fermi potentials and carrier densities increase, showing a strong impact of the photogeneration within the perovskite layer.
Nevertheless, even for unrealistically high $G_0 \gg 10^{1}$, all quantities remain bounded.

\noindent
\textbf{Three species model.}

\noindent

\begin{figure}[!ht]
    \includegraphics[width=\textwidth]{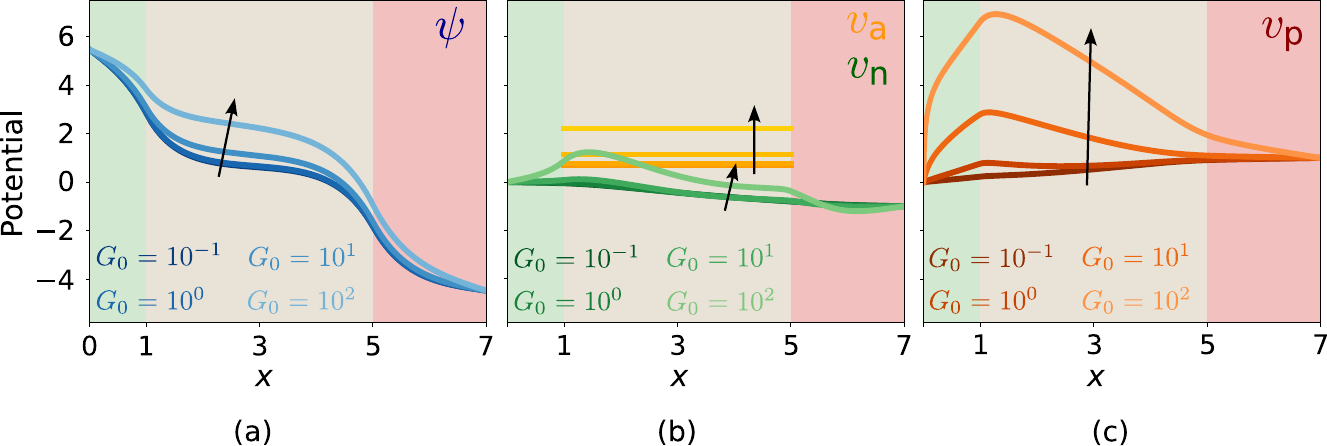}
    \caption{
        (a) Electric potential $\psi$, and quasi-Fermi potentials of (b) electrons $v_{\text{n}}$ (green), vacancies $v_\ions$ (gold) and (c) holes $v_{\text{p}}$, shown for different values of the photogeneration prefactor $G_0$.
        Brighter colors correspond to larger values of $G_0$ with arrows indicating the direction of increasing $G_0$.
        The applied voltage is set to $\overline{V} = 1$.
    }
    \label{fig:PSC-Potential-ions}
\end{figure}

Next, we recover the full model \eqref{eq:model-dimless} with $z_\ions = 1$.
We choose the Blakemore statistics, stated in \eqref{eq:Blakemore}, with the saturation density $S_\ions = 10$ as the statistics function.
As mentioned earlier, the vacancy potential $v_\ions$ is constant and determined by the integral equation in \eqref{eq:model-equilibrium ions}.
We set $M_\ions = | \Omega_{\text{PVK}} | C_{\text{a}}$, with $C = - C_\ions = -7.5$ in the perovskite layer $\Omega_{\text{PVK}}$ for this example, which enforces global charge neutrality, as expected from measurements \cite{courtier2018fast, Abdel2021Model}.

As before, \Cref{fig:PSC-Potential-ions} shows now the potential profiles for $\overline{V} = 1$ and varying $G_0$, with brighter colors indicating larger $G_0$ and arrows marking the direction of increase.
Lastly, \Cref{fig:PSC-Density-ions} presents the corresponding carrier densities and the $L^\infty$ norms of carrier densities and potentials as functions of $G_0$.
In addition, the figures now include the vacancy quasi-Fermi potential and the associated ionic vacancy density.

\begin{figure}[!ht]
    \includegraphics[width=\textwidth]{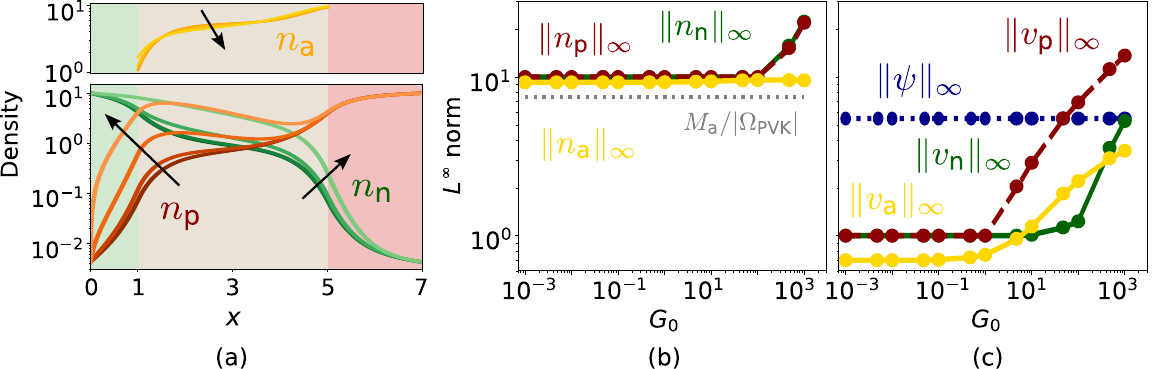}
    \caption{
        (a) Electron density (green), hole density (red), and vacancy density (gold) shown for different values of the photogeneration prefactor $G_0$.
        Brighter colors correspond to larger values of $G_0$ with arrows indicating the direction of increasing $G_0$.
        (b) $L^{\infty}$ norms of the carrier densities of electrons (green), holes (red dashed), and ions (gold) as well as
        (c) $L^{\infty}$ norms of the quasi-Fermi potentials of electrons (green), holes (red dashed), ions (gold) and the electric potential (blue dotted), plotted as a function of the photogeneration prefactor $G_{0}$.
        In (b) we see also the scaled total mass $M_\ions/| \Omega_{\text{PVK}}|$, which is the same for all $G_0$.
        The applied voltage is set to $\overline{V} = 1$.
    }
    \label{fig:PSC-Density-ions}
\end{figure}

The electron and hole potentials and densities remain nearly unchanged compared to the previous case, with only moderate variations for small $G_0$ and the strongest effect on the hole quasi-Fermi potential.
The additional ions flatten the electric potential in the intrinsic layer slightly, consistent with physical expectations.
The weak influence of $G_0$ on the ionic vacancy density is also reasonable, since photogeneration affects only electrons and holes, with vacancies coupled to them only indirectly through the electric potential.

Additional vacancies do not noticeably affect the $L^\infty$ norms.
For all relevant values of $G_0$, the potentials and carrier densities remain bounded by their respective Dirichlet boundary values.
The ionic vacancy density itself is constrained by construction via the state equation through the assumptions in \eqref{hyp:statistics-a} and, therefore, cannot exceed the saturation density $S_\ions = 10$.
These observations confirm that the extended model remains stable and consistent with the expected physical behavior.

\subsection{Laser-beam-induced-current technique} \label{sec:Sim-LBIC}

In the final example, we simulate the laser-beam-induced-current (LBIC) technique, a method used to detect inhomogeneities in classical semiconductors.
In this application, it is common to assume a bipolar charge transport model, i.e., $z_\ions = 0$.
The computational domain is divided into two doped regions, $\overline{\Omega} = \overline{\Omega}_{\text{n}} \cup \overline{\Omega}_{\text{p}}$, with $\Omega = (0,8) \times (0,4)$ and $\Omega_{\text{p}} = (2,6) \times (1,3)$, as shown in \Cref{fig:LBIC-Setup}a.
For the discretization, we employ a structured grid with $1219$ nodes, where the control volume sizes range between $5.0 \times 10^{-3}$ and $2.8 \times 10^{-2}$.
The boundary of the domain consists of a Dirichlet part $\Gamma^D=\Gamma^D_1\cup \Gamma^D_2$, with $\Gamma^D_1=\{0\}\times(0,4)$ and $\Gamma^D_2=\{8\}\times(0,4)$,  and a Neumann part $\Gamma^N$, where homogeneous boundary conditions are imposed.

The LBIC technique measures the outgoing current $I_{\Gamma^D_2}$  given by \eqref{outward:current} through the device's right ohmic contact $\Gamma^D_2$ (and well defined according to Corollary \ref{coro:bound:current}), while a laser beam is applied at a given position $(x_0, y_0) \in \overline{\Omega}$.
Repeating this procedure for multiple beam positions yields the LBIC signal.
The generation rate, induced by the laser, can be described by a Gaussian profile
\begin{align*}
G(x, y) = G_0 \exp \left( - \frac{(x-x_0)^2 + (y-y_0)^2}{2 \sigma^2}\right),
\end{align*}
with the given beam center $(x_0, y_0)$, amplitude $G_0$, and a fixed beam width $\sigma = 0.5$.
A visualization of the generation profile at $(x_0, y_0) = (4, 2)$ is provided in \Cref{fig:LBIC-Setup}b.

\begin{figure}[!ht]
    \includegraphics[width=\textwidth]{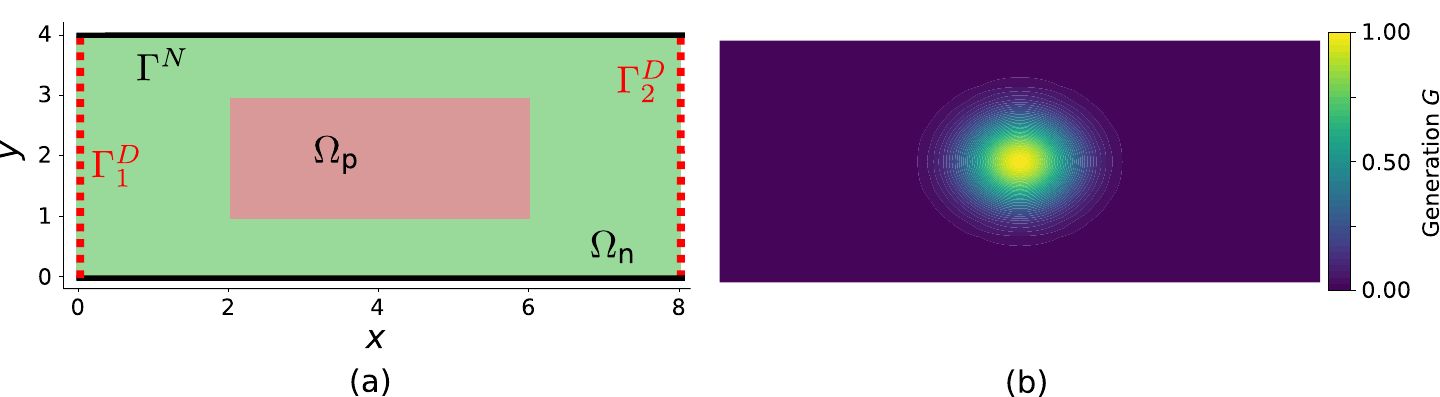}
    \caption{
        (a) Schematic of the p–n device, consisting of a p-doped material layer (red) embedded within an n-doped material layer (green).
        (b) Spatial distribution of the optical generation rate $G$, modeling a focused laser beam centered at $(x_0, y_0) = (4, 2)$, with parameters $G_0 = 1$ and $\sigma = 0.5$.
    }
    \label{fig:LBIC-Setup}
\end{figure}

\begin{figure}[!ht]
    \includegraphics[width=\textwidth]{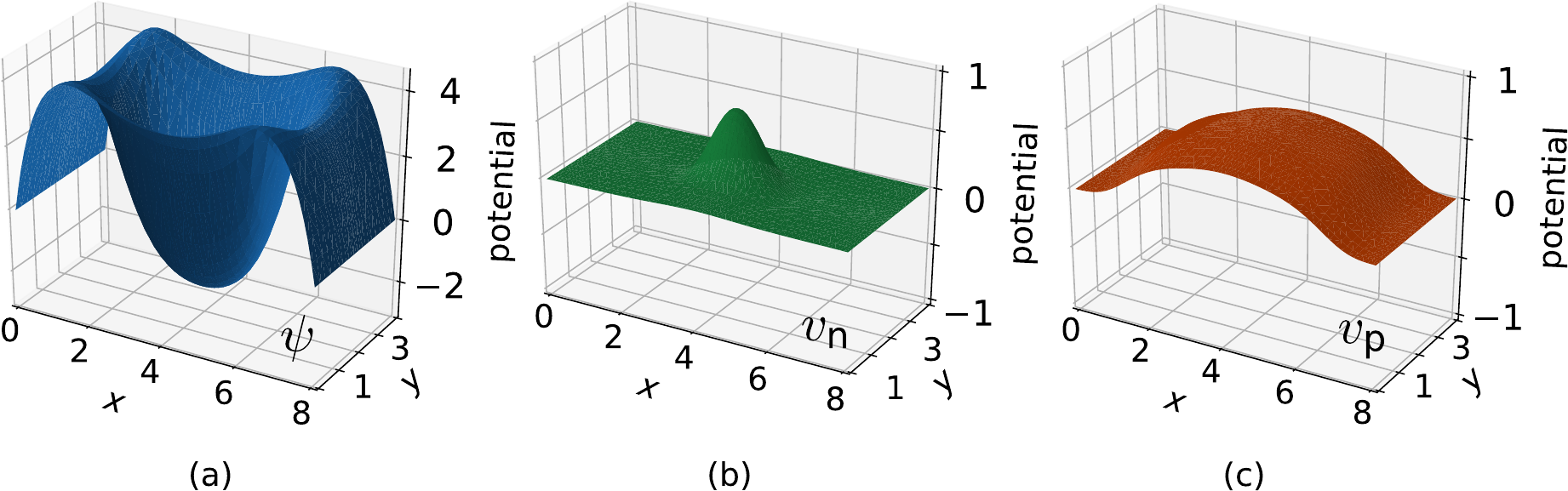}
    \caption{
        Three-dimensional surface plots of (a) the electrostatic potential $\psi$, (b) the electron quasi-Fermi potential $v_{\text{n}}$, and (c) the hole quasi-Fermi potential $v_{\text{p}}$ under illumination by a laser beam positioned at the center of the domain, i.e., at $(x_0, y_0) = (4, 2)$.
        The calculations correspond to parameters $G_0 = 1$, $\lambda = 1$, and $C_{\text{n}} = - C_{\text{p}} = 10$.
    }
    \label{fig:LBIC-Potentials}
\end{figure}

\Cref{fig:LBIC-Potentials} shows the potentials and \Cref{fig:LBIC-Density} the carrier densities for a Debye length $\lambda = 1.0$, a laser amplitude $G_0 = 1$, doping values $C = C_{\text{n}} = 10$ in $\Omega_\text{n}$, $C = -C_{\text{p}} = -10$ in $\Omega_\text{p}$, and a focused laser beam centered at $(x_0, y_0) = (4,2)$.
As the laser beam is positioned in the domain center, the potentials are mostly impacted in that area.
The carrier densities in \Cref{fig:LBIC-Density} also follow expectations: we have a higher hole density in $\Omega_\holes$ and a higher electron density in $\Omega_\electrons$, related to the choices of doping in that layers.

\begin{figure}[!ht]
    \includegraphics[width=\textwidth]{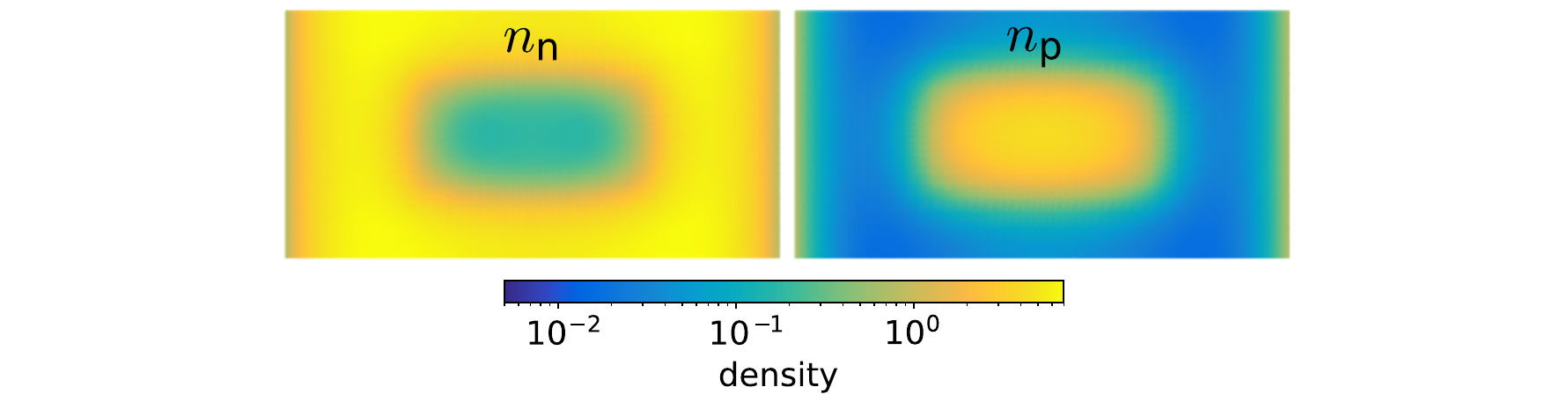}
    \caption{
        Contour plots of the electron $n_{\text{n}}$ and hole densities $n_{\text{p}}$ under illumination by a laser beam positioned at the center of the domain, i.e., at $(x_0, y_0) = (4, 2)$.
        The calculations correspond to parameters $G_0 = 1$, $\lambda = 1$, and $C_{\text{n}} = - C_{\text{p}} = 10$.
    }
    \label{fig:LBIC-Density}
\end{figure}

Next, we investigate how the total current changes in dependence of the laser beam position.
\Cref{fig:LBIC-Current}a shows the simulated two-dimensional LBIC signal.
The laser beam center $(x_0, y_0)$ was varied across all mesh points, and the total current was computed for each laser position.
The resulting signal behaves as expected, with a minimum just left of the p–n interface and a maximum to the right, reflecting the characteristic junction response to the focused laser beam.

Finally, we investigate how the Debye length $\lambda$, the doping magnitude $C$, and the beam amplitude $G_0$ influence the LBIC signal.
We fixed $y_0 = 2$ and computed the total  current for varying laser positions along the $x$-direction (following the red plane in \Cref{fig:LBIC-Current}a).

\begin{figure}[!ht]
    \includegraphics[width=\textwidth]{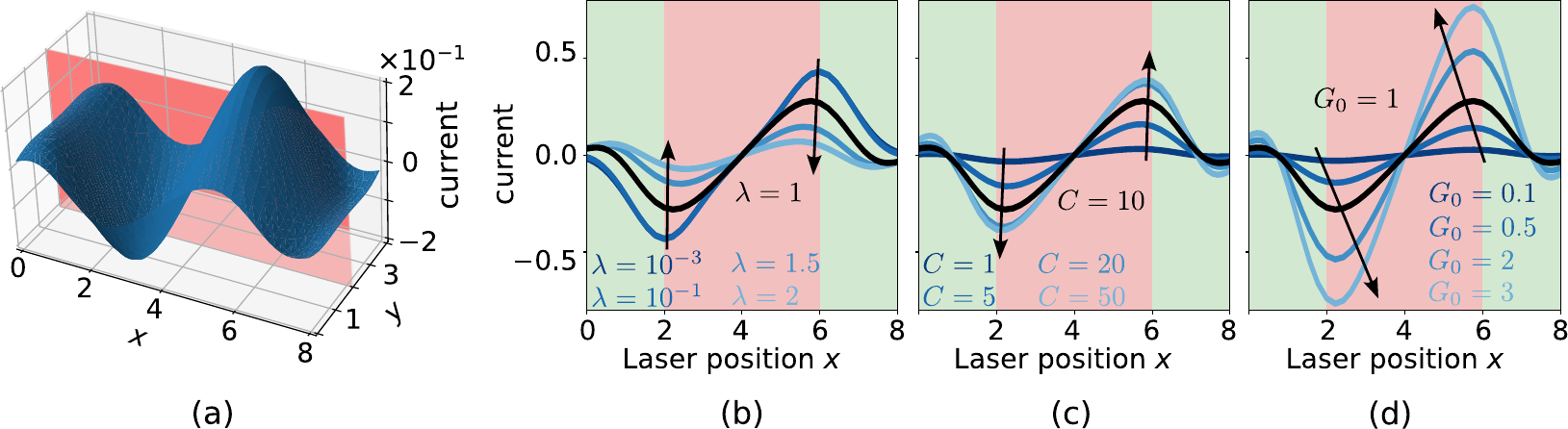}
    \caption{
        (a) Surface plot of the LBIC signal, represented by the total current, obtained by scanning the laser beam across the grid, that is, by sequentially repositioning the beam to different locations $(x_0, y_0)$ and computing the resulting current at each position.
        The calculations correspond to parameters $G_0 = 1$, $\lambda = 1$, and $C_{\text{n}} = - C_{\text{p}} = 10$.
        Furthermore, we show the one-dimensional cross-sections of the LBIC signal at $y_0 = 2$, i.e., along the red plane in (a), where the laser position is varied along the $x$-direction.
        For readability, we use the notation $C = C_\electrons = -C_\holes$.
        The signal is shown for:
        (b) $G_0 = 1$, $C = 10$ and varying $\lambda$,
        (c) fixed $G_0 = 1$, $\lambda = 1$ and varying $C$, and
        (d) fixed $\lambda = 1$, fixed $C = 10$ and varying $G_0$.
        Brighter colors correspond to larger values of the varying parameter, with arrows indicating the direction of increasing parameter values.
    }
    \label{fig:LBIC-Current}
\end{figure}

Non-destructive techniques such as the LBIC method generally correspond to inverse problems, in which the measured current signal is known, while the underlying material parameters, for instance, the doping profile, remain to be determined \cite{piani2024data}.
In such applications, the current signal is measured at different laser positions, and the objective is to identify the material parameters that reproduce this signal.
In the following, we address the forward problem and analyze how variations in key input parameters affect the total current.
\Cref{fig:LBIC-Current}b–d show one-dimensional cross-sections of the LBIC signal for different parameter values.
Brighter colors indicate larger values, with arrows showing the direction of increase, while the black curve corresponds to the one-dimensional signal from \Cref{fig:LBIC-Current}a.
Smaller Debye lengths lead to larger currents because they strengthen the internal electric field.
Increasing the doping raises the current by adding more carriers to the semiconductor, while a higher beam amplitude produces more photocarriers, also boosting the current.
In all cases, the overall shape of the current signal is bounded and preserved, independent of the absolute magnitude of these parameters.

\appendix
\section{Stampacchia's lemma}
For completeness we give a short proof of a slightly improved version of \cite[Lemma 4.1]{Stampacchia_1965}.
\begin{lemma}\label{lem:stamp}
Let $\mathcal{E}$ be a non-increasing and non-negative function satisfying, for some $x_0\in\mathbb{R}$, $\alpha, \zeta>0$ and $\beta>1$,
\[
\mathcal{E}(y)\leq \zeta\frac{\mathcal{E}(x)^\beta}{(y-x)^\alpha}\,, \quad \forall\, y>x\geq x_0.
\] 
Then  \[
\mathcal{E}(x)=0\,,\quad \forall\ x\geq x_0 + \zeta^{\frac1\alpha}\frac{\beta^{\frac{\beta}{\beta-1}}}{\beta-1}\mathcal{E}(x_0)^{\frac{\beta-1}{\alpha}}.\]
\end{lemma}
\begin{proof}
If $\mathcal{E}(x_0)=0$ there is nothing to prove. Otherwise, let $\varphi(x) = \mathcal{E}(x)/\mathcal{E}(x_0)$ and observe that it satisfies $\varphi(y)\leq \zeta'\varphi(x)^\beta(y-x)^{-\alpha}$ for $y>x\geq x_0$, with $\zeta' = \zeta\mathcal{E}(x_0)^{\beta-1}$. From there let $\rho\in(0,1)$, and define $x_{k+1} = x_k + (\zeta')^{\frac1\alpha}\rho^{\frac{(\beta-1)k-1}{\alpha}}$ for all $k\in\mathbb{N}$. By induction $0\leq\varphi(x_k)\leq\rho^k$ so letting $k\to\infty$, $\lim x_k = x_0+z(\rho)$ with $z(\rho)^\alpha = \zeta'\rho^{-1}(1-\rho^{(\beta-1)/\alpha})^{-\alpha}$ and yields $\varphi(x_0+z(\rho)) = 0$. Minimizing in $\rho$ yields the result. 
\end{proof}

\paragraph{Funding.}
AB, CCH and MH acknowledge the support of the CDP C2EMPI, together with the French State under the France-2030 programme, the University of Lille, the Initiative of Excellence of the University of Lille, the European Metropolis of Lille for their funding and support of the R-CDP-24-004-C2EMPI project.
DA was supported by the Deutsche Forschungsgemeinschaft (DFG, German Research Foundation) under Germany's Excellence Strategy -- The Berlin Mathematics Research Center MATH+ (EXC-2046/1, project ID: 390685689).
The research was conducted within the context of the Inria - WIAS Berlin ARISE Associate Team.

\bibliographystyle{abbrv}
\bibliography{lit}

\end{document}